\documentclass[preprint,12pt,authoryear,3p,times]{elsarticle}

\usepackage{amssymb}
\usepackage{amsmath}

\usepackage[normalem]{ulem}
\usepackage{comment}

\usepackage{cleveref}

\usepackage{xcolor}

\newtheorem{coro}{Corollary}
\newtheorem{lemm}{Lemma}
\newtheorem{assum}{Assumption}
\newtheorem{theorem}{Theorem}
\newtheorem{prop}{Proposition}
\newtheorem{proof}{Proof}

\DeclareMathOperator*{\argmin}{arg\,min}
\DeclareMathOperator*{\argmax}{arg\,max}

\journal{an international conference}

\begin{document}

\begin{frontmatter}



\title{A bottleneck model with shared autonomous vehicles: \\
Scale economies and price regulations} 


\author[label1]{Koki Satsukawa\corref{cor1}} 
\ead{satsukawa@staff.kanazawa-u.ac.jp}

\author[label2]{Yuki Takayama\corref{cor1}} 
\ead{takayama.y.cc65@m.isct.ac.jp}
\cortext[cor1]{Corresponding authors.}

\affiliation[label1]{organization={Institute of Transdisciplinary Sciences for Innovation, Kanazawa University},
            city={Kanazawa},
            postcode={920-1192}, 
            state={Ishikawa},
            country={Japan}}
            
\affiliation[label2]{organization={Department of Civil and Environmental Engineering, Institute of Science Tokyo},
        city={Meguro},
        postcode={152-8550}, 
        state={Tokyo},
        country={Japan}}

\begin{abstract}
This study examines how scale economies in the operation of shared autonomous vehicles (SAVs) affect the efficiency of a transportation system where SAVs coexist with normal vehicles (NVs).
We develop a bottleneck model where commuters choose their departure times and mode of travel between SAVs and NVs, and analyze equilibria under three SAV fare-setting scenarios: marginal cost pricing, average cost pricing, and unregulated monopoly pricing.
Marginal cost pricing reduces commuting costs but results in financial deficits for the service provider.
Average cost pricing ensures financial sustainability but has contrasting effects depending on the timing of implementation due to the existence of multiple equilibria:
when implemented too early, it discourages adoption of SAVs and increases commuting costs;
when introduced after SAV adoption reaches the monopoly equilibrium level, it promotes high adoption and achieves substantial cost reductions without a deficit.
We also show that expanding road capacity may increase commuting costs under average cost pricing, demonstrating the Downs--Thomson paradox in transportation systems with SAVs.
We next examine two optimal policies that improve social cost, including the operator's profit: the first-best policy that combines marginal cost pricing with congestion tolls, and the second-best policy that relies on fare regulation alone.
Our analysis shows that these policies can limit excessive adoption by discouraging overuse of SAVs.
This suggests that promoting SAV adoption does not always reduce social cost. 
\end{abstract}




\end{frontmatter}



\section{Introduction}

\subsection{\textcolor{black}{Background}}

The advent of autonomous vehicles is anticipated to revolutionize urban mobility in the near future. 
A widely discussed scenario involves the adoption of autonomous vehicles through \textit{shared mobility services} provided by transportation companies, rather than individual ownership \citep{narayanan2020shareda}. 
This paradigm shift is expected to improve the efficiency of transportation systems in which shared autonomous vehicles (SAVs) coexist with normal vehicles (NVs).

An essential feature of shared mobility is its inherent tendency to benefit from \textit{scale economies.} 
When vehicles are shared among users, fixed costs (e.g., vehicle acquisition, maintenance, and operational infrastructure) can be distributed across more users, which reduces the average cost per trip. 
Higher utilization rates of shared fleets enable providers to achieve greater operational efficiency, which in turn improves service quality and attracts more users.
While scale economies can enhance efficiency, they also lead to challenges similar to those studied in the context of public transport~\citep{ying2005sensitivity,iryo2019properties,small2024economics}, which may hinder the realization of the expected efficiency gains from this paradigm shift.

\color{black}
{\color{black}
Among these challenges, }
one important implication of scale economies is the potential emergence of \textit{natural monopoly}.
When fixed costs are high and scale economies are strong, a single provider can supply the services demanded by users at a lower cost than any combination of two or more providers.
This cost advantage enables the provider to dominate the market, naturally leading to a monopoly.
Natural monopolies are a common feature of public transport systems and shared mobility services \citep{horcher2021review}.
Without regulation, a monopolistic provider prioritizes profit maximization, resulting in higher fares for commuters and inefficiencies in service provision.
Therefore, effective fare regulations, such as marginal cost pricing or average cost pricing, are essential to mitigate the inefficiencies \citep{varian2024intermediate}.

{\color{black} In addition to these implications,} scale economies also introduce a \textit{feedback loop} between service adoption and operational efficiency.
As more users adopt a shared service, the average cost falls, which improves service quality and attracts even more users.
Conversely, a decline in adoption increases average costs, degrading service quality and further discouraging usage.  
Such feedback mechanisms can lead to multiple equilibria, where the system may settle into a high- or low-adoption equilibrium, depending on initial conditions.

The mechanisms can, in turn, give rise to paradoxical system responses, where seemingly beneficial interventions lead to unintended consequences.
A notable example is the classic Downs--Thomson paradox in public transport \citep{downs1962law,thomson1977great}.
In that context, improvements in road infrastructure reduce public transport ridership, raise its average cost, and degrade its quality.  
A similar dynamic may emerge even as SAVs and NVs share the same road space: as road capacity expands, some commuters may shift to NVs, reducing SAV utilization and increasing average costs. 
This, in turn, leads to a deterioration in service quality and higher total commuting costs.

\color{black}
Recent theoretical work has examined the impacts of SAVs on transportation systems. 
Among these, the bottleneck model has become a central analytical tool because it effectively captures key features of peak-period traffic congestion, which are essential for evaluating commuter behavior under limited road capacity.
Using this framework, many studies have focused on two main effects of SAVs: \textit{the capacity effect,} which increases road throughput by reducing headways, and \textit{the value-of-time (VOT) effect,} which reflects the time savings from allowing commuters to engage in other activities during in-vehicle travel.
These effects have been extensively analyzed in the literature on shared and autonomous mobility systems \citep[e.g.,][]{Van_den_Berg2016-mv,tian2019morning,yu2022will}.\footnote{Other modeling approaches have also been employed to analyze VOT and capacity effects during peak-period congestion. Notably, \citet{dantsuji2024hypercongestiona} adopt a bathtub model, which captures capacity drops that are not incorporated in bottleneck models, to evaluate how AVs influence congestion dynamics through these effects.}

\color{black}
Despite this growing body of work, relatively little attention has been paid to scale economies, which are a defining feature of shared mobility services.
This gap is important because scale economies directly influence how fare structures affect user adoption, operational efficiency, and overall system performance. 
Therefore, important interactions among pricing, adoption behavior, and system efficiency remain insufficiently understood in the current literature.

\subsection{\textcolor{black}{Contribution and overview}}

\color{black}
This study shows the influence of scale economies on the efficiency of transport systems with SAVs. 
To this end, we develop a bottleneck model that incorporates SAVs 
to examine how scale economies and natural monopoly characteristics affect commuters' mode choices and traffic congestion. 
Using this model, we analyze three SAV fare-setting scenarios: marginal cost pricing, average cost pricing, and unregulated monopoly pricing.
In addition, we evaluate two transport policies designed to achieve an efficient transportation system. 
The first-best policy combines 
marginal cost pricing with vehicle-specific, time-varying congestion tolls to address inefficiencies from both natural monopoly and congestion. 
The second-best policy assumes that congestion tolls are not feasible and focuses on fare regulation alone.

\color{black}
{\color{black}
The analysis reveals that fare-setting rules shape SAV adoption through distinct feedback mechanisms induced by scale economies.}
Marginal cost pricing lowers commuting costs by encouraging SAV usage, but it cannot sustain service provision in the presence of fixed costs, resulting in financial deficits.
Average cost pricing ensures cost recovery, but its impacts on system performance depend critically on the level of SAV usage at the time of implementation.
If average cost pricing is implemented when SAV adoption remains limited, the resulting high fares suppress usage and increase commuting costs.
By contrast, when fares are initially unregulated and the provider initially sets profit-maximizing fares, subsequent implementation of average cost pricing can sustain high adoption levels, improve service efficiency, and reduce commuting costs.
This contrast arises from a feedback loop induced by scale economies, whereby adoption levels and average costs mutually reinforce each other.

In addition, road capacity expansion can interact with fare regulation in counterintuitive ways.
Under average cost pricing, expanding road capacity may unintentionally increase total commuting costs.
This occurs because improved road conditions encourage commuters to switch back to NVs, which reduces SAV utilization and raises the average cost of providing the service.
As a result, the commuting cost faced by users increases, demonstrating the Downs--Thomson paradox in transport systems with SAVs.

The policy implications depend critically on the relative strength of two key characteristics of SAVs: the capacity effect and the VOT effect.
Under the first-best policy, SAV adoption is promoted when the capacity effect dominates, whereas it may be reduced when the VOT effect is stronger.
{\color{black} Under the second-best policy, where policy intervention is limited to fare regulation, welfare-improving outcomes may require limiting SAV usage through higher fares when the VOT effect dominates (i.e., excessive SAV adoption can arise without fare regulation).}
A key implication of the policy analysis is that marginal cost pricing does not universally minimize social cost in the absence of congestion tolls.
When the VOT effect dominates the capacity effect, average cost pricing or even unregulated monopoly pricing can yield lower social costs by discouraging excessive SAV adoption.
Overall, these results underscore the importance of aligning fare policy with the underlying technological characteristics of SAVs.

\subsection{\textcolor{black}{Related Literature}}

The implications of scale economies in public transport systems have long been studied in the economics and transport literature.\footnote{For a comprehensive review of the economics of public transport systems, including fare regulation and market structure, see \citet{horcher2021review}.
For a broad survey of the expected impacts, modeling approaches, and policy considerations surrounding SAVs, see \citet{narayanan2020shareda}.} 
Many theoretical studies have examined mode choice between private vehicles and public transport systems characterized by scale economies \citep[e.g.,][]{Tabuchi1993-xi,Danielis2002}.
{\color{black} 
A number of theoretical studies have examined the Downs--Thomson paradox in multimodal transport systems with rail transit \citep[e.g.,][]{arnott2000twomode, basso2012integrating, bell2012road, zhang2014downs,wang2019optimal}.}
In the context of bus transit systems, several theoretical studies have also incorporated scale economies \citep[e.g.,][]{cantarella2015daytoday, li2016dynamics, li2018traffic, Pandey2024-xc}.

A growing body of research has explored the implications of autonomous vehicles (AVs) for urban transportation systems, focusing on their potential to change commuter behavior and alleviate congestion. 
Building on \cite{Vickrey1969-rg}'s bottleneck model, numerous studies have examined AV impacts on departure time choice and mode competition. 
Early works considered the VOT effect and the capacity effect due to AVs \citep[e.g.,][]{Van_den_Berg2016-mv}, while later studies incorporated SAVs, parking constraints, and policy instruments such as dedicated lanes and reservation systems \citep[e.g.,][]{tian2019morning, lamotte2017use, li2022can}. 
Other extensions introduced in-vehicle activity utility \citep{pudane2020departure,Yu2022-dv,wu2023managing} and the competition and cooperation between different car manufacturers that may provide both NVs and AVs \citep{yu2022will}. 
These studies have significantly advanced our understanding of AV-induced behavioral shifts, but they do not explicitly address provider-side issues such as scale economies or fare regulation, which are central to SAV services.\footnote{
Recent studies have considered scale economies in shared mobility systems, particularly on-demand ridepooling services \citep{fielbaum2023economies, fielbaum2024are}.
These studies focus on how demand affects system-level cost structures through matching efficiency, waiting times, and detour lengths. However, they do not use bottleneck models or examine the policy implications of fare regulation or road capacity expansion. Our study complements this literature by extending a bottleneck model to examine how scale economies interact with fare regulations and pricing policies in shaping commuter behavior and system efficiency.}

This study develops a bottleneck model of shared autonomous vehicles that explicitly incorporates scale economies and natural monopoly.
By embedding a fixed cost into a multimodal departure time choice framework, the model enables a unified analysis of marginal cost pricing, average cost pricing, and unregulated monopoly pricing for SAV services.
This framework allows us to examine how fare-setting rules shape SAV adoption, congestion patterns, and system performance through feedback between usage levels and cost recovery.
By providing a transparent setting in which pricing regimes and infrastructure changes can be directly compared within a unified framework, the study offers new insights into the design of fare regulation for shared autonomous mobility under congestion.

\color{black}
The remainder of this paper is organized as follows.
Section 2 presents the bottleneck model with SAVs, incorporating scale economies and natural monopoly characteristics into commuters' departure time and mode choice decisions.
Section 3 analyzes three fare-setting scenarios—marginal cost pricing, average cost pricing, and monopoly pricing—and examines the interaction between fare regulation and road capacity expansion, highlighting the potential for the Downs--Thomson paradox in systems with SAVs.
Section 4 examines the effects of the first-best policy, while Section 5 evaluates those of the second-best policy.
Section 6 concludes by summarizing the main findings and suggesting directions for future research.

\begin{figure}[t]
	\begin{center}
	\hspace{0mm}
    \includegraphics[width=0.6\linewidth,clip]{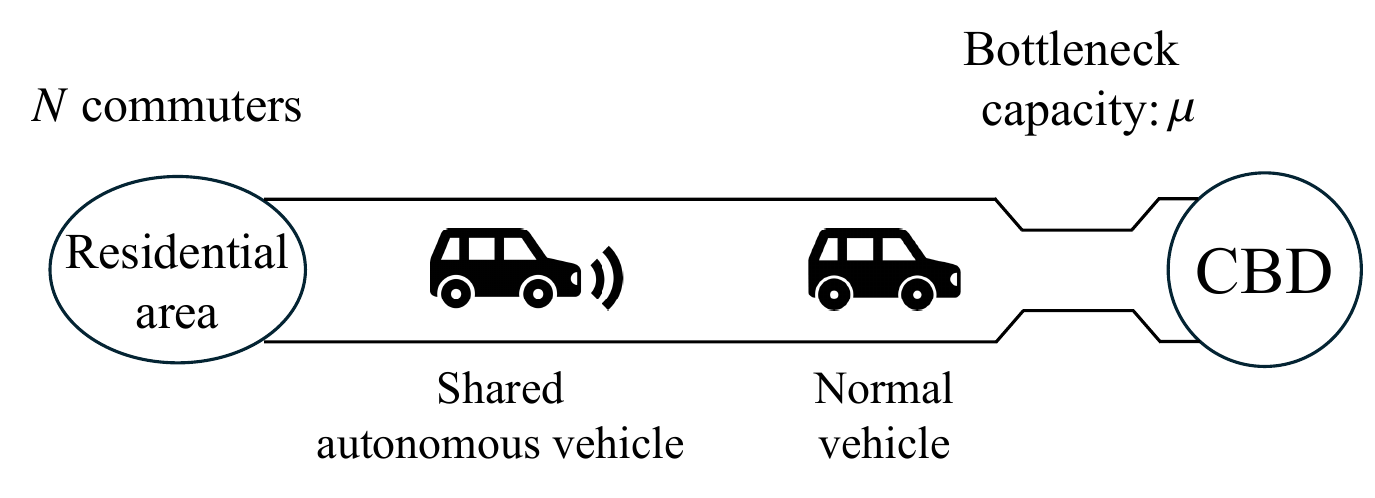}
	\end{center}
    \vspace{-2mm}
	\caption{Network and normal/autonomous vehicles}
    \vspace{-0mm}
    \label{Fig:Network}
\end{figure}

\section{Model settings}\label{Sec:Model}

\subsection{Network, commuters, and transport modes}
Consider a city consisting of a central business district (CBD) and a residential area connected by a single road (Figure~\ref{Fig:Network}). 
The road includes a bottleneck with capacity $\mu$ at its end. 
Queuing congestion is modeled as a point queue that obeys the first-in-first-out (FIFO) principle, consistent with standard bottleneck models~\citep[][]{Vickrey1969-rg,Hendrickson1981-bq,Arnott1990-oa,Arnott1993-jd}.
The free-flow travel time is constant and denoted by $t_f$.

A fixed number of commuters travel to the CBD along this road. 
They are treated as a continuum with total mass $N$. Commuters use either NVs or SAVs. The number of NV commuters is denoted by $N_{n}$, and that of SAV commuters is denoted by $N_{a}$.
\textcolor{black}{
NV commuters own their vehicles and incur a fixed cost term $F_{n}$. This term includes costs related to parking, such as parking fees. By contrast, SAVs are operated by a service provider. Thus, SAV commuters do not incur such costs, and instead pay a fare $p$ to the provider. In addition, using an SAV involves a pickup waiting time. For tractability, we represent its monetary equivalent by a constant $w$, in line with previous studies~\citep[e.g.,][]{Ma2017-vg,Tian2019-ug}.}\footnote{\textcolor{black}{In reality, pickup waiting time may depend on the provider's fleet operations (e.g., fleet sizing, dispatching, and repositioning). This study abstracts from these supply-side operational decisions and treats the pickup waiting time as an exogenous constant to keep the focus on the interaction between scale economies and commuters' demand-side choices. Extensions that endogenize fleet operations are briefly discussed in \Cref{Sec:Conclusion} as directions for future work.
	}}
	
Following \cite{Van_den_Berg2016-mv}, we assume that SAVs have two effects: a capacity effect and a VOT effect.
{\color{black}
For the capacity effect, we assume that the bottleneck capacity becomes $\mu/\kappa$ when an SAV passes through it. Here, $\kappa$ is an aggregate parameter that captures the overall capacity impact of SAV-related technologies, including connectivity and automation.{\interfootnotelinepenalty=10000 \footnote{\textcolor{black}{In mixed traffic, the capacity effects may depend on the penetration rate of connected-vehicle technology. However, the literature shows that heterogeneity in VOT can induce equilibrium sorting that generates temporally concentrated SAV flows in a bottleneck setting. We therefore assume that connectivity operates primarily within these flows, and summarize the resulting capacity effect using the reduced-form parameter $\kappa$.}}}
In this study, we focus on the case $0 < \kappa < 1$, which represents a forward-looking scenario in which mature automation and connectivity allow SAVs to operate safely at shorter headways, thereby making capacity improvements feasible.}
For the VOT effect, SAV commuters experience a reduction because they can engage in other activities during in-vehicle time. Let $\alpha$ denote the VOT for NV commuters for in-vehicle time. The VOT for SAV commuters is $\theta \cdot \alpha$, where $\theta$ $(0 < \theta < 1)$ is the VOT reduction parameter. Without loss of generality, we set $\alpha = 1$.


\subsection{Behavior of commuters and the service provider}
Each commuter chooses their departure time and transport mode to minimize the commuting cost, which consists of free-flow travel time, queuing delay, schedule delay, and mode-specific costs.  
We define $q(t)$ as the bottleneck queuing delay for commuters arriving at the destination (CBD) at time $t$.

The schedule-delay cost arises from the difference between actual and desired arrival times $t_{d}$.
We assume that $t_{d}$ is the same for all commuters, and normalize it to zero without loss of generality.
Following the standard bottleneck literature~\citep[e.g.,][]{Arnott1990-oa}, the cost is represented by a piecewise-linear function.
Let $s(t)$ denote the schedule-delay cost for commuters arriving at time $t\in\mathbb{R}$; it is given by 
\begin{align}
	s(t)=
	\begin{cases}
		-\beta t		\quad & \text{if}\quad t<0,\\
		\gamma t		\quad & \text{if}\quad t\geq 0,
	\end{cases}
\end{align}
where $\beta$ and $\gamma$ are the marginal schedule-delay penalties for early and late arrival, respectively.
We assume $\beta <\theta < 1$ so that the VOTs for all commuters exceed $\beta$.



Then, the commuting costs for each mode are expressed as follows:
\begin{align}
	&c_{n}(t ) = t_{f} + q(t) + s(t) + F_{n},\\
	&c_{a}(t ) = \theta \left\{ t_{f} + q(t) \right\} + s(t) + p  +  w,
\end{align}
where $c_{n}(t)$ and $c_{a}(t)$ represent the commuting costs for NV and SAV commuters arriving at the destination at time $t$, respectively.

The service provider operates SAVs and earns profit as the difference between revenue and total operating costs.
The profit $\pi$ is expressed as 
\begin{align}
    \pi = (p-m)N_{a} - F_{a},\label{Eq:Profit}
\end{align}
where $m$ is the marginal operating cost per commuter (e.g., energy and maintenance); $F_{a}$ is the fixed cost (e.g., fleet capital, software platform, and depots).

Given this profit formulation, the average operating cost is
\begin{align}
AC(N_{a}) = m + \cfrac{F_{a}}{N_{a}},
\end{align}
which is strictly decreasing in $N_{a}$; 
hence the service provider benefits from \textit{scale economies}.
Scale economies often lead to a natural monopoly, under which the provider may set the profit-maximizing fares.
This may result in an inefficient transportation system in which commuters face high commuting costs.
To mitigate such distortions, a public authority commonly introduces \textit{fare regulation}.



In this study, we compare the following three fare-setting scenarios to evaluate the impact of fare regulation on the NV--SAV transportation system.
\begin{itemize}
	\item \textit{Marginal cost (MC) pricing.}  
		The fare is set equal to the marginal operating cost:
		\begin{align}
			p=m.\label{Eq:MCP}
		\end{align}
		In standard microeconomic theory, MC pricing constitutes the first-best benchmark because it minimizes commuting costs.
		However, with a positive fixed cost $F_{a}$, this pricing yields negative profit.
		Therefore, MC pricing is applicable when a public authority operates the SAVs as a form of public transport and is prepared to subsidize the resulting deficit.


	\item \textit{Average cost (AC) pricing.} 
		The fare equals the average operating cost: 
		\begin{align}
			p=m + \cfrac{F_{a}}{N_{a}}.\label{Eq:AverageCostPricing}
		\end{align}
		AC pricing is a standard pricing rule that mitigates allocative inefficiency of a natural monopoly while ensuring the provider's profits remain non-negative.

	\item \textit{Unregulated monopoly pricing}.
		The service provider sets the fare to maximize profit $\pi$ given the demand function $N_{a}(p)$.
		Here, $N_{a}(p)$ denotes the demand for SAV service at fare $p$; its analytical form will be derived later as the equilibrium outcome of commuters' departure-time and mode choices.
		The fare $p$ satisfies the first-order condition of profit maximization:
		\begin{align}
			\cfrac{\partial \pi}{\partial p}=0\quad \Leftrightarrow \quad 
			N_{a}(p) + (p-m)\cfrac{\partial N_{a}(p)}{\partial p} = 0.\label{Eq:Cond_ProfitMax}
		\end{align}
		
\end{itemize}

\subsection{Equilibrium conditions}
We define an equilibrium as a state in which no commuter can reduce their commuting cost by unilaterally changing either their departure time or their travel mode.
Consistent with the existing literature~\citep[e.g.][]{Wu2014-xg,Van_den_Berg2016-mv}, the closed-form solution of the equilibrium is obtained by solving the two equilibrium conditions sequentially.
First, for any given mode-choice pattern $(N_{n}, N_{a})$, we characterize the departure-time choice equilibrium.
This equilibrium determines commuters' departure-time decisions and yields equilibrium commuting costs of each mode given that pattern.
Using these mode-specific commuting costs, we solve the mode-choice condition under which each commuter chooses the travel mode that minimizes their commuting cost.\footnote{This hierarchical structure reflects a two-stage decision process: commuters choose a travel mode first and then select a departure time. The procedure above solves these decisions backwards.}


\subsubsection{Departure-time choice condition}
Because the number of commuters using each mode is fixed, the departure-time choice equilibrium condition has essentially the same structure as the standard bottleneck model with heterogeneous VOT among commuters~\citep[e.g.,][]{Arnott1988-fy}.
Therefore, the equilibrium condition can be expressed as follows:
\begin{align}
	&
	\begin{cases}
		c_{n}(t) = c_{n}^{*}(N_{n},N_{a})\quad &\text{if}\quad n_{n}(t) > 0\\
		c_{n}(t) \geq c_{n}^{*}(N_{n},N_{a})\quad &\text{if}\quad n_{n}(t) = 0
	\end{cases}\quad \forall t\in\mathbb{R},\label{Eq:DTCcond_1}\\
	&
	\begin{cases}
		c_{a}(t) = c_{a}^{*}(N_{n},N_{a})\quad &\text{if}\quad n_{a}(t) > 0\\
		c_{a}(t) \geq c_{a}^{*}(N_{n},N_{a})\quad &\text{if}\quad n_{a}(t) = 0
	\end{cases}\quad \forall t\in\mathbb{R},\label{Eq:DTCcond_2}\\
	&
	\begin{cases}
		n_{n}(t) + \kappa n_{a}(t) = \mu\quad &\text{if}\quad q(t) > 0\\
		n_{n}(t) + \kappa n_{a}(t) \leq \mu\quad &\text{if}\quad q(t) = 0
	\end{cases}
	\quad \forall t\in\mathbb{R},\label{Eq:DTCcond_3}\\
	&\color{black}{\int_{t\in\mathbb{R}}n_{n}(t)\mathrm{d}t = N_{n},\quad \int_{t\in\mathbb{R}}n_{a}(t)\mathrm{d}t = N_{a}},\label{Eq:DTCcond_4}
\end{align}
where $n_{n}(t)$ and $n_{a}(t)$ denote the \textcolor{black}{arrival flow rates of NV and SAV commuters at the destination at time $t$, respectively.}
$c_{n}^{*}(N_{n},N_{a})$, $c_{a}^{*}(N_{n},N_{a})$ represent the mode-specific commuting costs in the departure-time choice equilibrium when the mode-choice pattern is $(N_{n},N_{a})$.

Conditions~\eqref{Eq:DTCcond_1} and \eqref{Eq:DTCcond_2} state the departure-time choice conditions for NV and SAV commuters, respectively.
Condition~\eqref{Eq:DTCcond_3} is the queueing delay condition at the bottleneck: when a queue is present, the total departure-flow rate at the bottleneck equals the bottleneck capacity $\mu$; otherwise, the departure-flow rate does not exceed capacity.\footnote{\textcolor{black}{While we do not model the congestion impact of SAV deadheading (empty-vehicle travel to the next pickup) for analytical tractability, it may be captured in reduced form as an effective reduction in bottleneck capacity (i.e., a lower effective $\mu$) if its impact is approximately stable over the peak period.}}
Owing to the capacity effect of SAVs, their effective departure-flow rate is scaled by $\kappa$.
Condition~\eqref{Eq:DTCcond_4} is the flow conservation condition for each mode.

\subsubsection{Mode choice condition}
Given the mode-specific commuting costs $c_{n}^{*}(N_{n},N_{a})$, $c_{a}^{*}(N_{n},N_{a})$, each commuter chooses either an NV or an SAV so as to minimize their own commuting cost. 
Accordingly, the mode-choice equilibrium condition is expressed as follows:
\begin{align}
	&
	\begin{cases}
	  c_n^*(N_n, N_a) = c_a^*(N_n, N_a) &{\rm if} \quad N_n > 0, ~N_a > 0, \\
	  c_n^*(N_n, N_a) \leq c_a^*(N_n, N_a) &{\rm if} \quad N_a = 0, \\
	  c_n^*(N_n, N_a) \geq c_a^*(N_n, N_a) &{\rm if} \quad N_n = 0, 
	\end{cases}
	\label{Eq:IntegratedEqui-Cond1}
	\\
	&
	N_n+N_a=N. \label{Eq:IntegratedEqui-Cond2}
\end{align}
Condition~\eqref{Eq:IntegratedEqui-Cond1} ensures no commuter can reduce their commuting cost by unilaterally changing modes; Condition~\eqref{Eq:IntegratedEqui-Cond2} is the flow conservation condition. 

\subsubsection{Stability condition}
Because the traffic state that satisfies the mode choice conditions is not necessarily unique, we examine the stability of equilibria if multiple equilibria arise.
Previous studies have shown that scale economies can lead to the existence of multiple equilibria~\citep[][]{Tabuchi1993-xi,Pandey2024-xc}.
In these circumstances, it is crucial to identify which equilibrium is likely to be realized by analyzing how the traffic state evolves as commuters adjust their choices~\citep[][]{Beckmann1956-vr}.

To define the stability of the equilibrium, we model the mode choice as a continuous-time evolutionary dynamic process.
Let $u$ denote a continuous time index measured in days, and let $(N_{n}(u), N_{a}(u))$ be the mode choice pattern at time $u$ where $N_{n}(u) + N_{a}(u) = N$.
Because of this identity, the system reduces to the single state variable $N_{a}(u)$, which evolves according to the following differential equation: 
\begin{align}
\cfrac{\mathrm{d}}{\mathrm{d}u}N_{a}(u) = V(N_{a}(u)),
\end{align}
where $V(N_{a})$ denotes the evolutionary dynamics that is Lipschitz continuous on the one-dimensional simplex $\Delta:=\{\,N_a\in[0,N]\,\}$.

Building on evolutionary game theory~\citep[e.g.,][]{Sandholm2010-ht}, we characterize stability under a broad class of dynamics that satisfy the following two properties, positive correlation (PC) and Nash stationarity (NS):
\begin{align}
	\text{(PC)}\quad &V(N_{a})\cdot (c_{n}^{*}(N_{n},N_{a}) - c_{a}^{*}(N_{n},N_{a})) > 0\ \text{whenever}\ V(N_{a})\neq 0\\
	\text{(NS)}\quad &V(N_{a})=0\ \text{if and only if $(N_{n},N_{a})$ is the equilibrium}.
\end{align}
The PC property requires that, whenever the mode choice pattern is out of rest, the covariance between the growth rate of the number of commuters choosing each mode and its payoff (the negative of its commuting cost) be positive.
The NS property requires that every rest point of the evolutionary dynamics be precisely an equilibrium point.
Specific examples include the best response~\citep[][]{Gilboa1991-xj}, the Brown-von Neumann-Nash~\citep[][]{Brown1951-my}, and the Smith dynamics~\citep[][]{Smith1984-ed}.

Under the evolutionary dynamics, we investigate the \textit{local asymptotic stability} of the equilibrium. 
An equilibrium is locally asymptotically stable when every trajectory that starts sufficiently close remains close and eventually converges to the equilibrium.
This means that the traffic state returns to the original equilibrium through the commuters' rational adjustment behavior even after slight deviations occur.
Hence, such an equilibrium can be regarded as a state likely to be realized.\footnote{A formal definition is provided in \Cref{Sec:App-StabilityAnalysis}, where the stability of multiple equilibria is analyzed.}

\section{Equilibrium analysis}\label{Sec:Equilibrium}

\subsection{Departure time choice equilibrium}

We first derive the departure time choice equilibrium. 
When $N_{n}$ and $N_{a}$ are exogenously fixed, our model follows that in \cite{Van_den_Berg2016-mv}.
The literature showed that the following temporal sorting property holds in the transport system:
\begin{lemm}[\cite{Van_den_Berg2016-mv}]
	In the departure-time choice equilibrium, SAV commuters arrive closer to their desired arrival time than NV commuters do.
\end{lemm}
\noindent This is because the VOT for SAV commuters is scaled by $\theta < 1$: the ratio of their marginal schedule cost to VOT, $\beta/(\theta \alpha)$, exceeds that for NV commuters, $\beta/\alpha$.
Consistent with classical bottleneck models, commuters with a higher effective schedule cost choose arrival times that are closer to their desired arrival time.

Using the temporal sorting property, the equilibrium commuting costs for NV and SAV are uniquely derived as follows:
\begin{align}
	&c_{n}^{*}(N_{n}, N_{a})= 
	\cfrac{\beta \gamma}{\mu(\beta + \gamma)}(N_{n} + \kappa N_{a}) + t_{f} + F_{n}\label{Eq:DepartEquiCost-normal},\\
	&c_{a}^{*}(N_{n}, N_{a})= 
	\cfrac{\beta \gamma}{\mu(\beta + \gamma)}(\theta N_{n} + \kappa N_{a}) + \theta t_{f} + p  + w.\label{Eq:DepartEquiCost-autonomous}
\end{align}

Eqs.~\eqref{Eq:DepartEquiCost-normal} and \eqref{Eq:DepartEquiCost-autonomous} show that the VOT reduction parameter $\theta$ affects only the commuting cost for SAV commuters, whereas the capacity-expansion parameter $\kappa$ influences the commuting costs for both NV and SAV commuters.
This difference arises because the capacity effect shortens the SAV rush-hour period, thereby shifting the destination arrival times of NV commuters.
Because of the temporal sorting property, the shorter the SAV rush-hour window, the closer NV commuters can arrive at their desired arrival time.
Consequently, when the capacity effect strengthens (i.e., either because $\kappa$ decreases or because the number of SAV commuters increases), NV schedule costs decrease, which reduces their commuting costs.

By contrast, a decrease in $\theta$ reduces the VOT for SAV commuters only; it reduces their free-flow travel time and queuing-delay costs but leaves NV commuters unaffected.
These asymmetric impacts of $\kappa$ and $\theta$ play an important role in the social cost analyses in \Cref{Sec:FirstBest} and \Cref{Sec:SecondBest}.

\subsection{Mode choice equilibrium}
We then derive the mode choice equilibrium using the mode-specific commuting costs \eqref{Eq:DepartEquiCost-normal} and \eqref{Eq:DepartEquiCost-autonomous} under the departure time choice equilibrium.
In the remainder of this section, we present closed-form equilibria for each fare-setting scenario introduced in Section \ref{Sec:Model} (see \Cref{Sec:App-ClosedFormSolution_DUE} for the derivation details).

\subsubsection{Marginal cost pricing}
We first focus on MC pricing by setting the fare equal to the marginal cost, $p=m$.
Substituting this condition into Eqs.~\eqref{Eq:DepartEquiCost-normal}--\eqref{Eq:DepartEquiCost-autonomous} and the equilibrium conditions \eqref{Eq:IntegratedEqui-Cond1}--\eqref{Eq:IntegratedEqui-Cond2}, we derive the closed-form mode choice equilibrium, as follows:
\begin{prop}
	Define
	\begin{align}
		A := \cfrac{\beta\gamma(1-\theta)}{(\beta+\gamma)\mu},\quad
		B := \theta t_{f} + m  + w - (t_{f} + F_{n}),
	\end{align}
	and suppose that 
	\begin{align}
		&0<B<AN\label{Eq:Cond_NonExtreme}.
	\end{align}
	Then, both NV and SAV modes are chosen in equilibrium, and the corresponding numbers are uniquely determined by
	\begin{align}
		N_{n}^{MC*} = \cfrac{B}{A},\quad N_{a}^{MC*}=\cfrac{AN-B}{A}.\label{Eq:Nums_MCPricing}
	\end{align}
	If condition~\eqref{Eq:Cond_NonExtreme} does not hold, all commuters choose a single mode: either all NVs or all SAVs.
\end{prop}
\noindent 
This proposition suggests that the share of SAV commuters increases in cities with heavy congestion and long travel times: $N_{a}$ increases when (i) the population size $N$ is large; (ii) the capacity $\mu$ is small or the free-flow travel time $t_{f}$ is long; and (iii) the penalty parameters $\beta$ and $\gamma$ are large.
This is because the VOT effect of SAVs increases the incentive to use them as travel times become longer.


Condition~\eqref{Eq:Cond_NonExtreme} rules out corner cases in which all commuters choose a single mode.
$B$ denotes the difference in mode-specific fixed costs when the fare is set to MC.
The inequality $B>0$ ensures that, even at this low fare, the fixed cost of an SAV exceeds that of an NV; hence, some commuters still choose NVs.
The condition $AN>B$ requires a population to be large enough for congestion delays: commuters have an incentive to choose SAVs.
As shown later, these two inequalities are necessary for both modes to be used not only in the benchmark case but also under AC pricing and unregulated monopoly pricing.
They are consistent with our aim of analyzing how fare regulations affect the bimodal transportation system.
We therefore impose the following assumption for the remainder of the paper:
\begin{assum}
	Parameters satisfy Condition~\eqref{Eq:Cond_NonExtreme}.	
\end{assum}

\subsubsection{Average cost pricing}
We next consider AC pricing, under which the fare equals the average cost, $p=AC(N_{a})$.
Combining this condition with the mode choice equilibrium conditions, we derive the following proposition, which shows that AC pricing may give rise to multiple equilibria:
\begin{prop}\label{Prop:Equi_ACP}
	Suppose that the following condition holds:
	\begin{align}
		(AN-B)^{2}-4AF_{a} \geq 0.\label{Eq:Cond_NonUnique}
	\end{align}
	Then, multiple equilibria exist, and the number of SAV commuters at each equilibrium is given as follows:{\color{black} 
    \begin{align}
		&N_{a0}^{AC*}=0,\\
		&N_{a1}^{AC*}=\cfrac{AN-B-K}{2A},\\
		&N_{a2}^{AC*}=\cfrac{AN-B+K}{2A},
	\end{align}
	\noindent where $K = \sqrt{(AN-B)^{2}-4AF_{a}}$.}
	If Condition~\eqref{Eq:Cond_NonUnique} does not hold, the equilibrium is unique, and the equilibrium number of SAV commuters is given by $N_{a0}^{AC*}=0$.
\end{prop}
\noindent We see that SAVs are used only when Condition~\eqref{Eq:Cond_NonUnique} holds.
This condition does not hold, for example, when $N$ is small or $F_{a}$ is large.
In such situations, the small population relative to the fixed cost increases the average operating cost, which in turn drives up the fare, making SAVs less attractive to commuters.

We next examine the stability of the multiple equilibria:
\begin{prop}
	Suppose that Condition~\eqref{Eq:Cond_NonUnique} holds with strict inequality.
    Then, the equilibria $N_{a0}^{AC*}$ and $N_{a2}^{AC*}$ are locally asymptotically stable; the equilibrium $N_{a1}^{AC*}$ is unstable.
    When equality holds, the two equilibria $N_{a1}^{AC*}$ and $N_{a2}^{AC*}$ coincide.
	The coincident equilibrium is unstable, and the equilibrium $N_{a0}^{AC*}$ is stable.
\end{prop}
\begin{proof}
	See \Cref{Sec:App-StabilityAnalysis}.
\end{proof}

\noindent This proposition shows that when Condition~\eqref{Eq:Cond_NonUnique} holds with strict inequality, the equilibria with the smallest and largest shares of SAVs are stable, whereas the equilibrium with the intermediate share is unstable.

The stability result suggests that SAV adoption may be hindered if AC pricing is implemented at its early stages.
Specifically, if AC pricing is implemented before the number of SAV commuters $N_{a}$ exceeds $N_{a1}^{AC*}$, the system converges to the stable zero-adoption equilibrium, $N_{a0}^{AC*}$.
The mechanism lies in the positive feedback effect created by AC pricing between the number of SAV commuters and their commuting costs.
Once the number of SAV commuters falls below the unstable equilibrium $N_{a1}^{AC*}$, the commuting cost for SAV commuters exceeds that for NV commuters, i.e., $c_{n}^{*} < c_{a}^{*}$.
This is because under AC pricing, fewer commuters raise the regulated fare, thereby increasing the SAV commuting cost.
Hence, given evolutionary dynamics $V(N_{a})$ satisfying the PC property, commuters have an incentive to switch from SAVs to NVs and no incentive to switch back.
As a result, the number of SAV commuters keeps falling, and the system ultimately converges to the stable equilibrium $N_{a0}^{AC*}$ in which no one uses SAVs.

Conversely, if AC pricing is implemented after $N_{a}$ exceeds $N_{a1}^{AC*}$, the system converges to the high-adoption equilibrium $N_{a2}^{AC*}$.
This occurs because the positive feedback effect drives $N_{a}$ upward: as $N_{a}$ increases, the fare decreases, reducing the commuting cost for SAV commuters below that for NV commuters, i.e., $c_{n}^{*} > c_{a}^{*}$.
NV commuters thus have an incentive to switch to SAVs under the evolutionary dynamics, which further increases $N_{a}$.
Therefore, the timing of implementing AC pricing should be carefully coordinated with the current level of SAV adoption.

\begin{figure}[t]
	\begin{center}
	\hspace{0mm}
    \includegraphics[width=0.8\linewidth,clip]{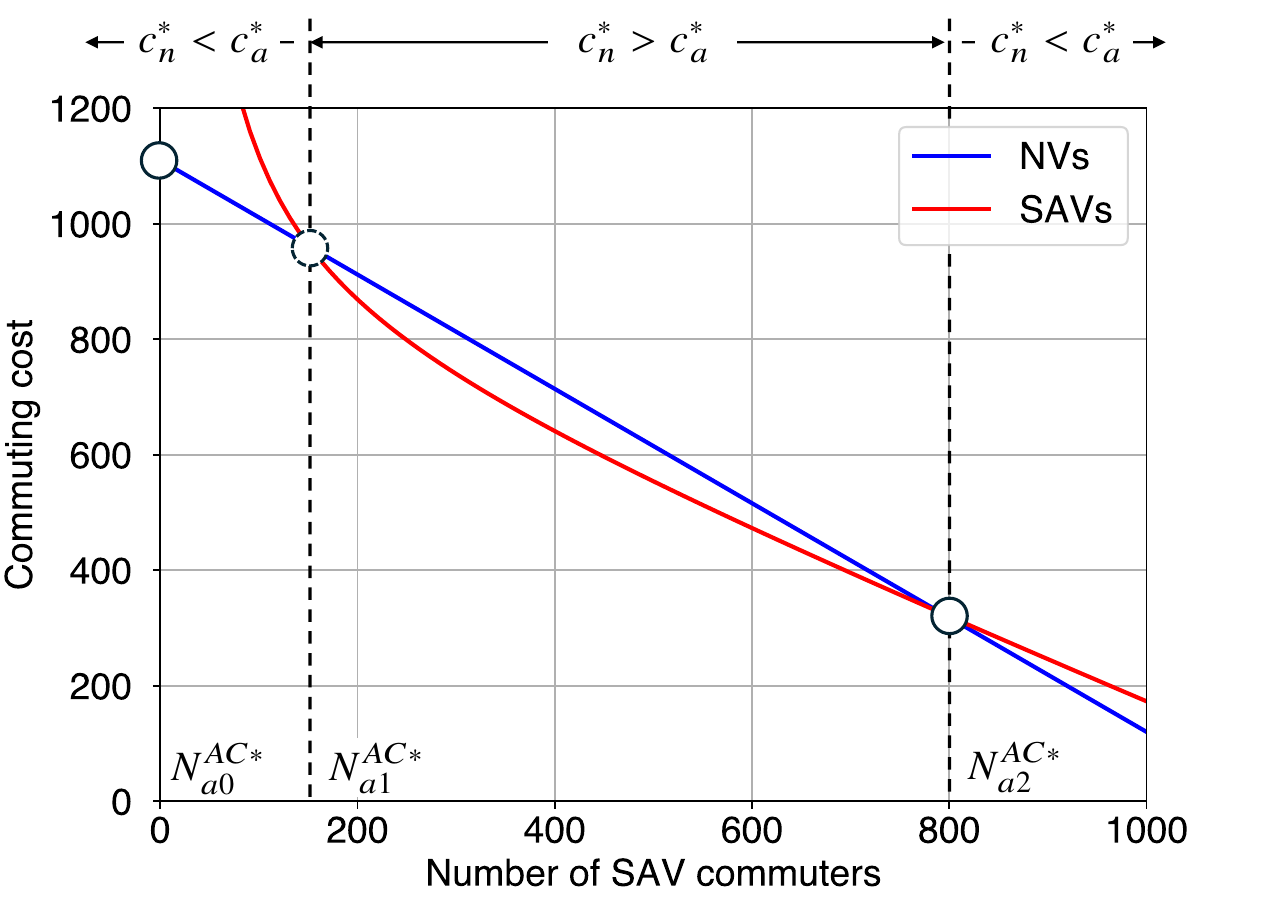}
	\end{center}
    \vspace{-2mm}
	\caption{Mode-specific commuting costs and the number of SAV commuters; solid and dashed circles represent the stable and unstable equilibria, respectively.
	\textcolor{black}{$N=1000$, $\kappa = 0.01$, $\theta = 0.7$, $\beta = 0.4$, $\gamma = 0.4$, $\mu = 0.2$, $t_{f} = 10$, $F_{n} = 100$, $m=100$, $w = 20$, $F_{a}=36000$}.
    }
    \vspace{-0mm}
    \label{Fig:CostCompare}
\end{figure}

We can graphically confirm the stability result from Figure~\ref{Fig:CostCompare}, which shows the relationship between the mode-specific commuting costs and $N_{a}$.
Since commuters choose transport modes to minimize their commuting costs, $N_{a}$ increases if $c_{n}^{*}>c_{a}^{*}$, meaning that the traffic state moves to the right in the figure.
Conversely, if $c_{n}^{*}<c_{a}^{*}$, $N_{a}$ decreases and the state moves to the left.
We thus confirm that traffic states converge to either $N_{a0}^{AC*}$ or $N_{a2}^{AC*}$, depending on whether $N_{a}$ is below or above the threshold $N_{a1}^{AC*}$.

\subsubsection{Unregulated monopoly pricing}
We lastly consider a monopoly situation.
When commuters use both transport modes, the demand function $N_{a}(p)$ is obtained by solving the equation $c^{*}_{n}(N_{n},N_{a}) = c^{*}_{a}(N_{n},N_{a})$, as follows:
\begin{align}
	N_{a}(p) = N - \cfrac{B+p-m}{A}.\label{Eq:DemandFunction}
\end{align}
By substituting this into the profit maximization condition~\eqref{Eq:Cond_ProfitMax}, the optimal fare for the service provider is expressed as:
\begin{align}
	p^{m*}=m + \cfrac{AN-B}{2}\label{Eq:Fare_Monopoly}.
\end{align}

We then derive the equilibrium number of SAV commuters $N^{m*}_{a}$ by combining the demand function with the optimal fare.
Substituting $N^{m*}_{a}$ and $p^{m*}$ into \eqref{Eq:Profit} yields the condition under which the service provider makes a profit.
These results are summarized in the following proposition:
\begin{prop}
    Suppose that Condition~\eqref{Eq:Cond_NonUnique} holds.
	Then, the profit $\pi$ in \eqref{Eq:Profit} becomes non-negative at the equilibrium.
    The equilibrium number of SAV commuters is expressed as follows:
    \begin{align}
    	N^{m*}_{a} = \cfrac{AN-B}{2A}.\label{Eq:AutoNum_Monopoly}
    \end{align}
\end{prop}

\noindent Interestingly, the condition under which profit becomes non-negative is identical to Condition~\eqref{Eq:Cond_NonUnique}.
This means that profit becomes positive only when SAVs are used under average cost pricing, such as when $N$ is large or $F_{a}$ is small.
Therefore, a sufficiently large market size is also crucial for promoting the service in an economically viable manner.

When Condition~\eqref{Eq:Cond_NonUnique} does not hold, the service provider will incur a deficit and withdraw from the market.
We therefore assume that $N_{a}^{m*}=0$ in this case.

%
%

\subsection{Comparison of equilibria under different fare-setting scenarios}\label{Sec:Equilibrium_Compare}

\subsubsection{Mode choice and cost patterns}
This section compares equilibrium outcomes under three fare-setting scenarios.
We first establish the following theorem on the ordering of the number of SAV commuters and associated commuting costs across the three scenarios:
\begin{theorem}\label{Theo:CompareEqui}
	Suppose that Condition~\eqref{Eq:Cond_NonUnique} holds.
	Then, the equilibrium number of SAV commuters under each fare-setting scenario satisfies the following relationship:
	\begin{align}
		N_{a}^{MC*}>N_{a2}^{AC*}\geq N_{a}^{m*}\geq N_{a1}^{AC*}>N_{a0}^{AC*}=0.\label{Eq:Compare_Number}
	\end{align} 
	Similarly, the equilibrium commuting cost satisfies,
	\begin{align}
		c^{MC*}<c_{2}^{AC*}\leq c^{m*}\leq c_{1}^{AC*}<c_{0}^{AC*}.
        \label{Eq:Compare_Cost}
	\end{align}
    The subscripts in the notations represent the corresponding equilibria.
    When Condition~\eqref{Eq:Cond_NonUnique} holds with strict inequality, the relationships \eqref{Eq:Compare_Number} and \eqref{Eq:Compare_Cost} also hold with strict inequalities.

    Suppose that Condition~\eqref{Eq:Cond_NonUnique} does not hold.
    Then, the following relationships hold:
    \begin{align}
        &N_{a}^{MC*}>N_{a}^{m*}=N_{a0}^{AC*}=0,\\
        &c^{MC*}<c^{m*}=c_{0}^{AC*}.
    \end{align}
\end{theorem}


\textbf{Theorem~\ref{Theo:CompareEqui}} shows that commuting costs decrease as the number of SAV commuters increases.
It further shows that MC pricing maximizes the SAV share and is therefore most beneficial for commuters; however, it imposes financial deficits on the service provider.
When Condition~\eqref{Eq:Cond_NonUnique} holds, both AC pricing and unregulated monopoly pricing can reduce commuting costs compared with those in the no-SAV case, while keeping the provider's profit non-negative.
In particular, the number of SAV commuters at the high-adoption equilibrium under AC pricing, $N_{a2}^{AC*}$, exceeds the number at the unregulated monopoly equilibrium, $N_{a}^{m*}$.
Therefore, when the high-adoption equilibrium is reached, AC fare regulation achieves a larger reduction in commuting costs while maintaining the provider's financial balance.


Interestingly, the theorem suggests that, with an appropriate timing of AC pricing, the system can converge to the high-adoption equilibrium $N_{a2}^{AC*}$ from any initial state.
Based on the theorem and the stability analysis under AC pricing, we propose the following two-step pricing strategy for the case in which Condition~\eqref{Eq:Cond_NonUnique} holds strictly:
\begin{description}
\item[\textbf{Step 1}] \textit{No fare regulation.}
Keep the fare unregulated and allow monopoly pricing.

\item[\textbf{Step 2}] \textit{Introduce AC fare regulation.}
Once the traffic state reaches the monopoly equilibrium, implement AC pricing.
\end{description}

\noindent This strategy overcomes a drawback of AC pricing---its inability to promote SAV adoption from a low adoption level---by temporarily allowing unregulated monopoly pricing.
As scale economies lead to a natural monopoly, the service provider sets the fare at $p^{m*}$, and the number of SAV commuters converges to $N^{m*}_{a}$.
As shown in Eq.~\eqref{Eq:Compare_Number}, $N^{m*}_{a}$ exceeds $N_{a1}^{AC*}$, the unstable equilibrium under AC pricing.
Implementing AC pricing at this point therefore shifts the system toward the high-adoption equilibrium $N_{a2}^{AC*}$, further increasing the SAV adoption and reducing commuting cost.
Consequently, this pricing strategy, contingent on the diffusion stage, facilitates widespread SAV adoption in an economically viable manner.





\subsubsection{Effects of capacity expansion}

{\color{black}
We next examine how an increase in bottleneck capacity $\mu$ affects equilibrium commuting cost when both NVs and SAVs are used.
The sensitivity of the equilibrium cost $c^{*}$ with respect to $\mu$ can be obtained by differentiating Eq.~\eqref{Eq:DepartEquiCost-normal} as follows:
\begin{align}
\cfrac{\mathrm{d} c^{*}}{\mathrm{d} \mu} = \cfrac{\beta\gamma}{(\beta+\gamma)\mu}\left[-\cfrac{N-(1-\kappa)N_{a}}{\mu}- (1-\kappa)\cfrac{\mathrm{d} N_{a}}{\mathrm{d} \mu}\right].\label{Eq:CapacityExpansion}
\end{align}

\noindent The first term represents the direct effect of capacity expansion and is always negative: for a given mode share, expanding capacity reduces congestion and the equilibrium commuting cost. The second term captures an indirect effect: the change in the equilibrium commuting cost shifts commuters' mode choices, i.e., the equilibrium number of SAV commuters. This changes the aggregate capacity effect induced by SAVs, thereby further affecting the equilibrium commuting cost. 

Under MC pricing and unregulated monopoly pricing, capacity expansion always reduces the equilibrium commuting cost, even though it induces a mode shift from SAVs to NVs. Increasing capacity reduces congestion-related delays for both modes; however, because NV commuters have a higher VOT than SAV commuters, the resulting cost reduction is larger for NV commuters. This induces some commuters to switch from SAVs to NVs, implying $\mathrm{d}N_a/\mathrm{d}\mu<0$ and a positive indirect-effect term in Eq.~\eqref{Eq:CapacityExpansion}. Nevertheless, this indirect effect never outweighs the direct effect in both fare-setting scenarios, and thus the total derivative remains negative. This can be confirmed by substituting the closed-form solutions in Eqs.~\eqref{Eq:Nums_MCPricing} and \eqref{Eq:AutoNum_Monopoly} into Eq.~\eqref{Eq:CapacityExpansion}:
\begin{alignat}{3}
    \text{[MC pricing]}\quad
    &\cfrac{\mathrm{d} N^{MC*}_{a}}{\mathrm{d}\mu} = -\cfrac{B}{A \mu}
    \quad &\Rightarrow\quad
    &\cfrac{\mathrm{d} c^{MC*}}{\mathrm{d}\mu}
     = -\cfrac{\beta \gamma}{\beta + \gamma}\cfrac{\kappa N}{\mu^{2}} < 0.\\
    \text{[Monopoly pricing]}\quad
    &\cfrac{\mathrm{d} N^{m*}_{a}}{\mathrm{d}\mu} = -\cfrac{B}{2A \mu}
    \quad &\Rightarrow\quad
    &\cfrac{\mathrm{d} c^{m*}}{\mathrm{d}\mu}
     = -\cfrac{\beta \gamma}{\beta + \gamma}\cfrac{(1+\kappa) N}{2\mu^{2}} < 0.
\end{alignat}

By contrast, under AC pricing, the derivative in Eq.~\eqref{Eq:CapacityExpansion} can be positive: the \textit{Downs--Thomson paradox}, where capacity expansion increases commuting costs, can occur, particularly in the high-adoption equilibrium. Substituting the equilibrium number of SAV commuters under AC pricing into Eq.~\eqref{Eq:CapacityExpansion} yields the following theorem.

\begin{theorem}
Under AC pricing, suppose that Condition~\eqref{Eq:Cond_NonUnique} holds with strict inequality and consider the stable high-adoption equilibrium $N_{a2}^{AC*}$.
The equilibrium commuting cost increases as capacity increases when the following condition holds:
\begin{align}
 K < \left( AN-B-\cfrac{2F_{a}}{N} \right) \cfrac{1-\kappa}{1+\kappa}.\label{Eq:ParadoxCondition}
\end{align}
\end{theorem}

\begin{figure}[t]
	\centering
	\hspace{0mm}
    \includegraphics[width=0.8\linewidth,clip]{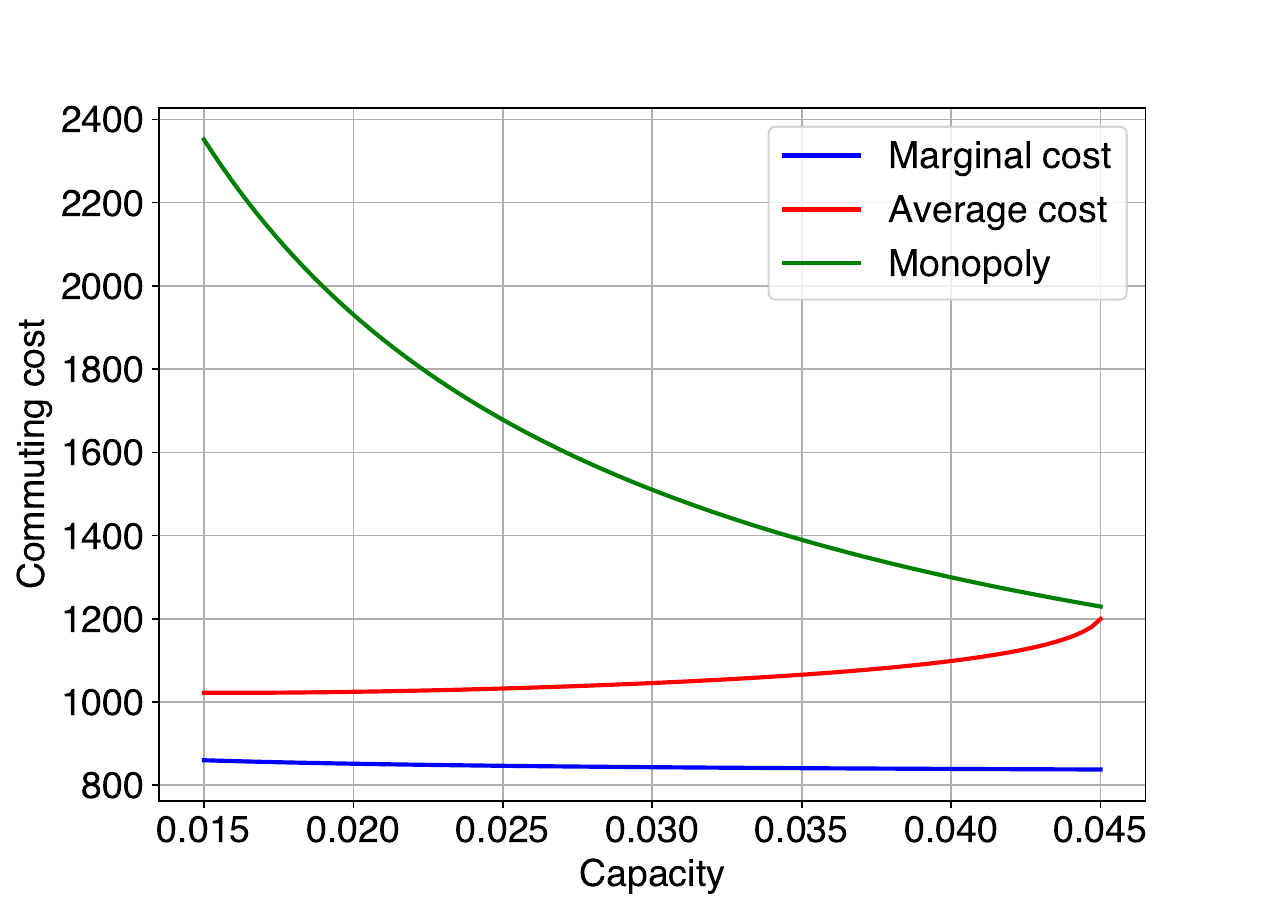}
    \vspace{-2mm}
    \caption{Downs--Thomson paradox under AC pricing. \textcolor{black}{$N=250$, $\kappa = 0.01$, $\theta = 0.7$, $\beta = 0.4$, $\gamma = 0.4$, $t_{f} = 10$, $F_{n} = 500$, $m=499$, $w = 100$, $F_{a} = 10500$}.}
    \vspace{-0mm}
    \label{Fig:Paradox}
\end{figure}

The paradox can arise because, in the high-adoption equilibrium under AC pricing, capacity expansion can trigger a positive feedback loop that amplifies the shift from SAVs to NVs. As under MC and monopoly pricing, capacity expansion makes NVs relatively more attractive, reducing $N_{a}$ and making the indirect term positive. Under AC pricing, this reduction in $N_{a}$ raises the fare $p=m+F_{a}/N_{a}$, which further discourages SAV use and reinforces the shift to NVs. As $N_{a}$ decreases, the aggregate capacity effect induced by SAVs weakens. Then, the indirect effect can outweigh the direct effect of capacity expansion, increasing the equilibrium commuting cost.
Figure~\ref{Fig:Paradox} confirms the occurrence of this paradox.

Eq.~\eqref{Eq:ParadoxCondition} indicates that the paradox is more likely to arise when (i) $\kappa$ is small and (ii) $K$ is small. When $\kappa$ is small, SAV commuters have a strong capacity advantage. In that case, a shift from SAV to NV can substantially worsen congestion because a decline in the SAV share causes a large reduction in the aggregate capacity effect induced by SAVs.

Second, when $K$ is small, the high-adoption equilibrium number of SAV commuters, $N_{a2}^{AC*}$, becomes highly sensitive to the bottleneck capacity $\mu$. This can be seen from the derivative with respect to $\mu$: 
\begin{align}
\cfrac{\mathrm{d} N_{a2}^{AC*}}{\mathrm{d}\mu} = - \cfrac{N_{a2}^{AC*} (N- N_{a2}^{AC*})}{\mu(N_{a2}^{AC*} - N_{a1}^{AC*})} = - \cfrac{AN_{a2}^{AC*} (N- N_{a2}^{AC*})}{\mu K} 
\end{align}
A smaller $K$ therefore tends to magnify $\left|\mathrm{d} N_{a2}^{AC*}/\mathrm{d} \mu\right|$,
strengthening the indirect effect.

To clarify this mechanism, we examine how much the equilibrium number of SAV commuters must adjust to restore $\Delta c\equiv c_{n}^{*} - c_{a}^{*}$ to zero after a capacity change creates a cost difference between the two modes. Differentiating $\Delta c \equiv c_{n}^{*} - c_{a}^{*}$ with respect to $N_a$ yields:
\begin{align}
\cfrac{\mathrm{d} \Delta c}{\mathrm{d} N_{a}} = \cfrac{F_{a}}{(N_{a})^{2}} - A.
\end{align}
This derivative decomposes the effect of $N_{a}$ on the mode-cost difference $\Delta c$ into two terms. The first term $F_{a}/(N_{a})^{2}$ captures the fare effect under AC pricing: as $N_{a} $ increases, the fare $p=m + F_{a}/N_{a}$ decreases, which reduces $c_{a}^{*}$ and thereby increases $\Delta c$. The second term $-A$ captures the congestion and schedule-delay effect: as $N_{a}$ increases (and $N_{n}$ decreases), NV commuters face smaller congestion and schedule-delay costs, which decreases $\Delta c$.
Importantly, the net effect of these two terms around the equilibrium is governed by $K$. Evaluating the derivative at the high-adoption equilibrium yields:
\begin{align}
\left. \cfrac{\mathrm{d} \Delta c}{\mathrm{d} N_{a}} \right|_{N_{a} = N_{a2}^{AC*}} = -\cfrac{K}{N_{a2}^{AC*}}
\end{align}
Hence, when $K$ is close to zero, the two effects nearly offset each other and $\Delta c$ becomes locally insensitive to changes in $N_{a}$. As a result, a capacity-induced change in relative costs requires a large adjustment in $N_{a}$ to restore $\Delta c = 0$, implying that $\left|\mathrm{d} N_{a2}^{AC*}/\mathrm{d} \mu\right|$ can be large. This strengthens the indirect effect and makes the paradox more likely.

}

These findings underscore the importance of integrating hard and soft transport policies to provide commuters with convenient and financially sustainable transport options.
Expanding capacity is not inherently inappropriate for improving traffic conditions under unregulated monopoly pricing; however, it may be ultimately ineffective if AC pricing is later implemented.
Infrastructure development should be planned with a long-term perspective that considers the potential direction of future transport policies.

\section{First-best policy}\label{Sec:FirstBest}
So far, we have examined how commuters' equilibrium costs vary across three fare-setting scenarios.
We now shift our attention to social costs that incorporate the service operator's profit, and to transport policies designed to achieve an efficient transportation system.

This section examines the first-best policy that minimizes the social cost.
Specifically, we assume that, in addition to fare regulation, a policymaker can impose time-varying congestion tolls that completely internalize congestion externalities.
The equilibrium condition is formulated under the first-best policy.
We then derive the equilibrium flow and cost patterns and analyze their properties.

\subsection{Formulation}
First, we define the first-best optimum, the outcome that the first-best policy aims to achieve, as the state that minimizes social cost.
Social cost is defined as the total commuting costs borne by commuters, net of the operator's profit.
Because queuing-delay costs are pure deadweight losses, the first-best optimum is obtained as the solution to the following linear programming:
\begin{align}
	\min_{\{n_{n}(t),n_{a}(t)\}_{t\in\mathbb{R}}}\quad &
	\begin{aligned}[t]
		SC(\{n_{n}(t),n_{a}(t)\}) :=\;&
		\int_{\mathbb{R}}\Big[
		n_{n}(t)\big(t_{f}+s(t)+F_{n}\big) \\
		&\qquad\qquad
		+\, n_{a}(t)\big(\theta t_{f}+s(t)+m  + w \big)
		\Big]\,\mathrm{d}t \;+\; F_{a}
	\end{aligned}
	\label{Eq:FirstBest_Object}\\
	\text{s.t.}\quad
	&\int_{\mathbb{R}}n_{n}(t)\,\mathrm{d}t + \int_{\mathbb{R}}n_{a}(t)\,\mathrm{d}t = N,
	\label{Eq:FirstBest_Const1}\\
	&n_{n}(t) + \kappa n_{a}(t)\leq \mu,\quad \forall t\in\mathbb{R},
	\label{Eq:FirstBest_Const2}\\
	&n_{n}(t)\geq 0,\quad n_{a}(t)\geq 0,\quad \forall t\in\mathbb{R}.
	\label{Eq:FirstBest_Const3}
\end{align}

\noindent The objective function represents the social cost when queuing is completely eliminated.
Constraint \eqref{Eq:FirstBest_Const1} is the flow conservation condition.
Constraint \eqref{Eq:FirstBest_Const2} is the capacity constraint.
Constraint \eqref{Eq:FirstBest_Const3} ensures non-negativity on the flow rates.

We then consider a first-best policy that achieves this outcome in equilibrium.
The policy consists of optimal fare regulation together with a time-varying congestion toll.
Let $\tau(t)$ denote the congestion toll paid by NV commuters who arrive at their destination at time $t$.
We assume that each SAV commuter pays $\kappa \tau(t)$, because an SAV imposes congestion externalities that are $\kappa$ times those generated by an NV commuter.\footnote{Charging differentiated congestion tolls by vehicle class is technically feasible and has already been implemented in practice.
An alternative is to establish, outside the transport system, a market for tradable access rights or credits that allows
commuters to pre-purchase bottleneck permits priced according to their vehicle classes.}
We further assume that the SAV fare is set according to MC pricing: $p=m$.
Under these assumptions, the commuting costs for NV and SAV commuters arriving at time $t$ are given by
\begin{align}
	&c_{n}^{FB}(t) = s(t) + t_{f} + F_{n} + \tau(t),\\
	&c_{a}^{FB}(t) = s(t) + \theta t_{f} + m  + w + \kappa \tau(t).
\end{align}
The equilibrium conditions under the first-best policy are expressed as follows:
\begin{align}
	&
	\begin{cases}
		c_{n}^{FB}(t) = c^{FB*}\quad &\text{if}\quad n_{n}(t) > 0\\
		c_{n}^{FB}(t) \geq c^{FB*} \quad &\text{if}\quad n_{n}(t) = 0
	\end{cases}\qquad \forall t\in\mathbb{R},\label{Eq:FirstBest_Optimality1}\\
	&
	\begin{cases}
		c_{a}^{FB}(t) = c^{FB*}\quad &\text{if}\quad n_{a}(t) > 0\\
		c_{a}^{FB}(t) \geq c^{FB*} \quad &\text{if}\quad n_{a}(t) = 0
	\end{cases}\qquad \forall t\in\mathbb{R},\label{Eq:FirstBest_Optimality2}\\
	&
	\begin{cases}
		n_{n}(t) + \kappa n_{a}(t) - \mu = 0\quad &\text{if}\quad \tau(t) > 0\\
		n_{n}(t) + \kappa n_{a}(t) - \mu \leq 0\quad &\text{if}\quad \tau(t) = 0
	\end{cases}\qquad \forall t\in\mathbb{R},\label{Eq:FirstBest_Optimality3}\\
	&
	\int_{t\in\mathbb{R}}n_{n}(t)\mathrm{d}t + \int_{t\in\mathbb{R}}n_{a}(t)\mathrm{d}t = N,\label{Eq:FirstBest_Optimality4}
\end{align}
where $c^{FB*}$ represents the equilibrium commuting cost under the first-best policy.
Complementarity conditions~\eqref{Eq:FirstBest_Optimality1} and \eqref{Eq:FirstBest_Optimality2} represent the equilibrium conditions for departure-time and mode choices.
Condition~\eqref{Eq:FirstBest_Optimality3} is the bottleneck capacity constraint, ensuring that the optimal toll eliminates the bottleneck queue.

Examining the optimality conditions of \eqref{Eq:FirstBest_Object}--\eqref{Eq:FirstBest_Const3} shows that the proposed policy, i.e., MC pricing combined with vehicle-specific congestion tolls, achieves the first-best outcome.
The equilibrium conditions coincide with the first-order conditions when the Lagrange multiplier for constraint~\eqref{Eq:FirstBest_Const1} is set to $c^{FB*}$ and the multiplier for constraint~\eqref{Eq:FirstBest_Const2} is set to $\tau(t)$.
Consequently, an equilibrium that satisfies the equilibrium conditions attains the first-best optimum, and the proposed policy is therefore first best.

\subsection{Characterization of the equilibrium under the first-best policy}
We here characterize the equilibrium under the first-best policy (hereinafter referred to as the first-best equilibrium) based on the equilibrium conditions~\eqref{Eq:FirstBest_Optimality1}--\eqref{Eq:FirstBest_Optimality4}.
By analyzing them, we first obtain the following lemma, which shows that the temporal-sorting property observed in the equilibrium (without congestion tolls) is preserved even in the first-best equilibrium:
\begin{lemm}
	In the first-best equilibrium, SAV commuters arrive closer to their desired arrival time than NV commuters do.
\end{lemm}
\begin{proof}
	See \Cref{Sec:App-TemporalSortingFB}.	
\end{proof}

This lemma shows that SAV commuters choose their departure times so that they arrive during periods of lower schedule costs, i.e., periods subject to higher congestion tolls.
Because the toll imposed on an SAV is lower than that on an NV, SAV commuters have a stronger incentive than NV commuters to select such periods.
From the perspective of optimal allocation, the underlying mechanism is the capacity effect of SAVs: because an SAV allows more commuters to pass through the bottleneck per unit time, assigning SAVs to periods with lower schedule costs reduces the social cost.

Based on this temporal-sorting property, the first-best equilibrium can be derived by considering the following two cases (See \Cref{Sec:App-ClosedFormSolution_FirstBest} for the derivation details):

\subsubsection*{Case (i) mixture of normal and shared autonomous vehicles}
Suppose that
\begin{align}
	N > \cfrac{B}{1-\kappa}\cfrac{(\beta + \gamma)\mu}{\beta \gamma}.\label{Eq:FirstBest_PopulationThres}
\end{align}
Then, the equilibrium number of SAV commuters and the equilibrium commuting cost are expressed as follows:
\begin{align}
	&N_{a}^{FB*} = N - \cfrac{(1-\theta)B}{(1-\kappa)A},\\
	&c^{FB*} = \cfrac{\kappa AN}{1-\theta} + B  + t_{f} + F_{n}.
\end{align}

\subsubsection*{Case (ii) normal vehicles only}
If Condition~\eqref{Eq:FirstBest_PopulationThres} does not hold, the equilibrium number of SAV commuters and the equilibrium commuting cost are given by
\begin{align}
	&N_{a}^{FB*} = 0,\\
	&c^{FB*} = \cfrac{AN}{1-\theta}+ t_{f} + F_{n}.
\end{align}

The results indicate that a flow pattern consisting exclusively of SAVs cannot arise under the first-best equilibrium.
The reason lies in the condition $B>0$, which represents that the marginal cost of using an SAV exceeds that of an NV.
Because queues are eliminated in the first-best equilibrium, commuters' mode choice depends only on differences in toll savings $(1-\kappa)\tau(t)$ and marginal costs $B$.
During periods with low tolls (e.g., immediately after the bottleneck becomes active or just before it becomes inactive), the higher marginal cost of using an SAV outweighs the toll savings.
As a result, commuters have an incentive to choose NVs.

\textcolor{black}{An SAV becomes optimal when the toll savings exceed the marginal cost disadvantage, i.e., when $\tau(t) > B/(1-\kappa)$. Because the optimal toll is the shadow price of the capacity constraint, its peak level increases with the population size $N$: when $N$ is larger, the toll around the desired arrival time must increase to keep the bottleneck queue-free. Condition (46) gives the minimum $N$ such that the peak toll exceeds $B/(1-\kappa)$; when it holds, Case (i) arises and SAVs are used in the peak-neighborhood.}


\begin{figure}[t]
	\centering
	\hspace{0mm}
    \includegraphics[width=0.8\linewidth,clip]{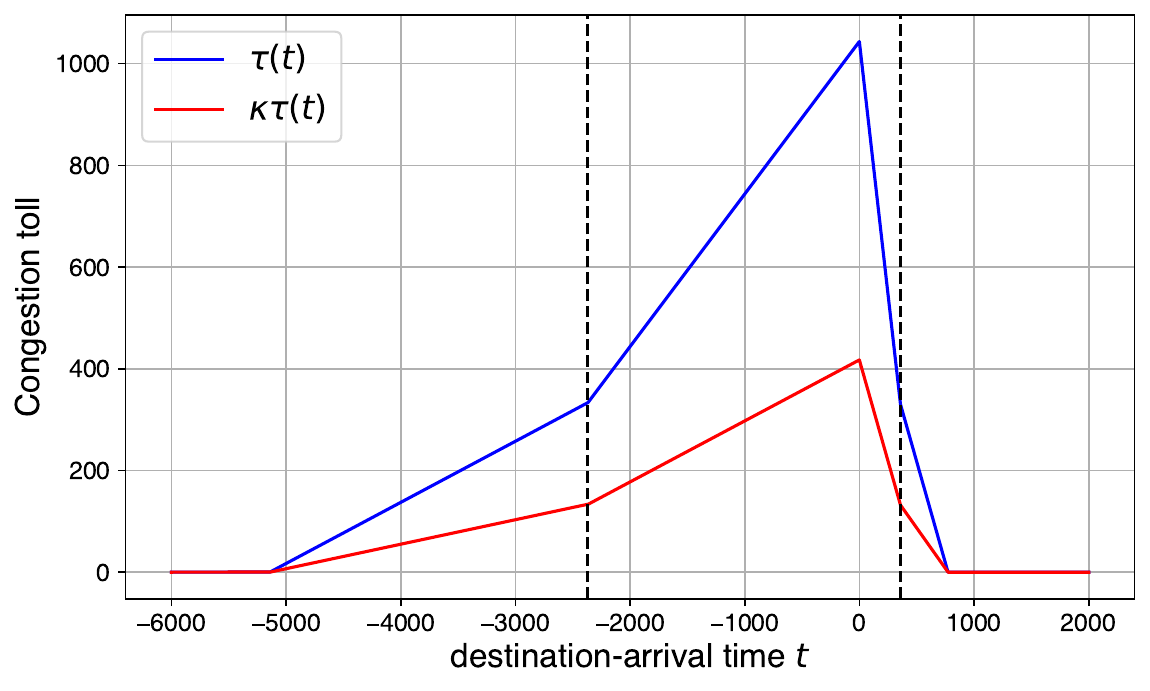}
    \vspace{-2mm}
    \caption{Time-varying congestion tolls in the first-best equilibrium. \textcolor{black}{$N=10000$, $\theta = 0.5$, $\beta = 0.3$, $\gamma = 2.0$, $\mu = 1$, $t_{f} = 2$, $F_{n} = 500$, $m=800$, $w = 200$.}}
    \vspace{-0mm}
    \label{Fig:CumulativeFB}
\end{figure}

\Cref{Fig:CumulativeFB} illustrates the time-varying congestion tolls $\tau(t)$ for NV commuters and $\kappa\tau(t)$ for SAV commuters.
In the figure, the tolls peak around the desired arrival time because SAVs are used during this period. 
An SAV commuter pays $\kappa \tau(t)$, which is $\kappa$ times the toll paid by an NV commuter and therefore lower.
Taking this into account, the policymaker sets $\tau(t)$ sufficiently high to prevent queues from forming among SAV commuters.
In addition, a comparison of $\tau(t)$ and $\kappa \tau(t)$ shows that their gap is large when SAVs are used and narrows as the destination arrival time moves away from the desired arrival time. 
Consequently, in periods where this gap is small, commuters have an incentive to choose NVs, whose marginal cost is lower than that of SAVs.

\subsection{Pareto-improvement condition}
Using the closed-form solution, we examine whether the first-best policy achieves a Pareto improvement, i.e., whether it reduces the social cost without increasing any commuter's cost.
To identify a sufficient condition, we compare the commuting cost $c^{FB*}$ in the first-best equilibrium with that under MC pricing $c^{MC*}$, which is the lowest commuting cost among the three toll-free fare-setting scenarios.


For this analysis, the following capacity-VOT index $\eta$, which summarizes the relative strength of capacity and VOT effects, plays a central role: 
\begin{align}
	\eta := \cfrac{1-\kappa}{1-\theta} > 0,\qquad (\because 0 < \kappa<1,0 < \theta<1)
\end{align}
Its value is determined by the relative magnitudes of $\kappa$ and $\theta$:
\begin{align}
	\eta =
	\begin{cases}
		1 & (\kappa=\theta),\\
		>1 & (\kappa<\theta),\\
		<1 & (\kappa>\theta).
	\end{cases}
\end{align}
A lower $\kappa$ indicates a stronger capacity effect, while a lower $\theta$ indicates a larger VOT effect; hence, $\eta > 1$ implies that the capacity effect exceeds the VOT effect, while $\eta < 1$ implies the opposite.
Intuitively, a higher $\eta$ indicates stronger congestion mitigation from capacity expansion, making the widespread adoption of SAVs socially beneficial.

Using the index $\eta$, we can obtain the condition under which a Pareto improvement is attainable.
Let $c^{FB*}_{1}$ and $c^{FB*}_{2}$ denote the first-best equilibrium commuting cost in Cases (i) and (ii), respectively.
We then have the following proposition: 
\begin{prop}
	If $\eta \geq 1$, the first-best equilibrium necessarily consists of both NV and SAV commuters; that is, case (i) arises.
	Moreover, the first-best policy minimizes the social cost without increasing the individual commuting cost; that is, the following relationship holds: 
	\begin{align}
		c^{FB*}_{1} \leq c^{MC*}<c_{2}^{AC*}\leq c^{m*}<c_{0}^{AC*}.
	\end{align}
	Therefore, a Pareto improvement is achieved from the equilibrium that arises under any fare-setting scenario.
	If $\eta < 1$, the first-best policy minimizes the social cost, but it increases the commuting costs, at least relative to the equilibrium under MC pricing.	
\end{prop}
\begin{proof}
	See \Cref{Sec:App-ParetoImprovement}.
\end{proof}

\noindent Thus, when the capacity effect dominates the VOT effect, the social cost can be reduced while reducing the commuters' costs.

The key mechanism lies in the gap between the equilibrium numbers of SAV commuters under the first-best policy and MC pricing.
Let us consider Case (i).
By comparing $c^{FB*}_{1}$ and $c^{MC*}$, we obtain the following equation:
\begin{align}
	c^{MC*} - c^{FB*}_{1} = \eta A (N_{a}^{FB*} - N_{a}^{MC*}).
\end{align}
This equation indicates that whether $c^{FB*}_{1}$ is smaller than $c^{MC*}$ depends on whether the number of SAV commuters in the first-best equilibrium exceeds that in the equilibrium under MC pricing. 
Moreover, the difference in the number of SAV commuters between these two equilibria satisfies the following relationship:
\begin{align}
	N_{a}^{FB*} - N_{a}^{MC*}
	= \int_{t_{a}^{-}}^{t_{a}^{+}}n_{a}(t)\mathrm{d}t - N_{a}^{MC*}
	= \cfrac{B}{A}\left(1 - \cfrac{1}{\eta}\right)
	= 
	\begin{cases}
			>0	\quad & \text{if}\quad \eta > 1\\
			0	\quad & \text{if}\quad \eta = 1\\
			<0	\quad & \text{if}\quad \eta < 1
		\end{cases}
\end{align}
This relationship shows that when the capacity effect dominates ($\eta > 1$), $N_{a}^{FB*}$ exceeds $N_{a}^{MC*}$.
It follows that the first-best policy requires an additional increase in the number of SAV commuters. 
As the SAV share increases, the total capacity effect grows.
This reduces commuters' schedule costs and hence their commuting costs.

By contrast, when $\eta < 1$, the first-best policy reduces the number of SAV commuters.
Consequently, the total capacity effect in the first-best equilibrium decreases, and schedule costs may increase.
A similar phenomenon arises in Case (ii), where SAVs are not used: the first-best policy eliminates SAV usage, which increases commuting costs.
These results suggest that congestion tolls should be introduced, taking into account the strength of the capacity effect, i.e., the technological maturity of SAVs.

\subsection{Self-financing principle}
We demonstrated that the first-best equilibrium can be achieved by combining MC pricing with a time-varying congestion toll. 
For a congestion-pricing scheme to be socially acceptable, it is desirable to redistribute the toll revenue to commuters. 
As a benchmark form of redistribution, we consider using the revenue to finance the expansion of bottleneck capacity.
We then investigate whether the self-financing principle~\citep[][]{Mohring1962-wz} holds, namely, whether the revenue generated by the optimal congestion toll covers the cost of the optimal capacity level.

Let $K(\mu)$ denote the investment cost required to provide a bottleneck capacity of $\mu$.
We assume that this cost function is homogeneous of degree one, i.e.,
\begin{align}
	K(\mu) = \cfrac{\mathrm{d} K(\mu)}{\mathrm{d} \mu}\mu.
\end{align}

\noindent Adding this investment cost, we consider the following optimization problem in which a policymaker chooses the capacity $\mu$ and the flow profiles $\{n_{n}(t), n_{a}(t)\}_{t\in\mathbb{R}}$ simultaneously to minimize the social cost:
\begin{align}
	\min_{\{n_{n}(t),n_{a}(t)\},\mu}\quad &SC(\{n_{n}(t),n_{a}(t)\}) + K(\mu)\\
	&\text{s.t.}\quad 
	\eqref{Eq:FirstBest_Const1},
	\eqref{Eq:FirstBest_Const2},
	\eqref{Eq:FirstBest_Const3},
	\ \text{and}\ \mu\geq 0.
\end{align}

\noindent The optimality conditions consist of \eqref{Eq:FirstBest_Optimality1}--\eqref{Eq:FirstBest_Optimality4} and the complementarity condition for capacity:
\begin{align}
	\begin{cases}
		 \cfrac{\mathrm{d} K(\mu)}{\mathrm{d} \mu}  - \int_{t\in\mathbb{R}}\tau(t)\mathrm{d}t= 0\quad &\text{if}\quad \mu > 0,\\
		 \cfrac{\mathrm{d} K(\mu)}{\mathrm{d} \mu}  - \int_{t\in\mathbb{R}}\tau(t)\mathrm{d}t\geq 0\quad &\text{if}\quad \mu = 0.
	\end{cases}
\end{align}
Because $\mu=0$ prevents any commuter from reaching the destination, this case is clearly infeasible.
Formally, one can impose a sufficiently large penalty cost on $\mu=0$; then, the optimality condition reduces to the following identity:
\begin{align}
	K(\mu) =\mu \int_{t\in\mathbb{R}}\tau(t)\mathrm{d}t.
\end{align}

This optimality condition implies that the self-financing principle holds in the first-best equilibrium.
The left-hand side is the investment cost required to provide the optimal capacity $\mu$.
The right-hand side represents the total toll revenue collected in equilibrium under optimal congestion pricing. 
When NV commuters flow, each commuter whose destination arrival time is $t$ pays the time-varying toll $\tau(t)$, and the flow rate is $\mu$; 
When SAV commuters flow, each commuter pays $\kappa \tau(t)$, but the flow rate is $\mu/\kappa$.
Therefore, the instantaneous revenue at $t$ is both $\mu\tau(t)$ and integrating the revenue over time yields the right-hand side.
We thus derive the following proposition:
\begin{prop}
	If the investment-cost function $K(\mu)$ is homogeneous of degree one, the optimal investment cost equals the total revenue generated by the optimal congestion toll.	
\end{prop}

\noindent This proposition confirms that no public subsidy is needed to finance capacity expansion, because toll revenue is sufficient to cover the investment cost.

\section{Second-best policy}\label{Sec:SecondBest}
In the preceding section, we examined a setting in which the policymaker could eliminate queuing congestion via optimal congestion pricing. 
However, congestion pricing is not always socially acceptable, so departure times may remain uncontrolled.

This section examines the second-best policy when congestion externalities cannot be internalized.
We assume that the policymaker can regulate only the SAV fare and chooses it to minimize the social cost in equilibrium, subject to the departure-time and mode choice conditions.
We first derive the second-best fare and the corresponding number of SAV commuters.
We then compare these results with those under the three fare-setting scenarios examined earlier and identify which scenario yields the lowest social cost under which conditions.

\subsection{Formulation}
We first define the second-best optimum, the outcome that the second-best policy aims to achieve.
Under the equilibrium without a congestion toll, commuters experience the mode-specific commuting costs $c_{n}^{*}(N_{n}, N_{a})$, $c_{a}^{*}(N_{n}, N_{a})$ under the departure-time choice equilibrium according to their mode choices.
Hence, the total commuting cost can be written as follows:
\begin{align}
	\int_{ t\in \mathbb{R}}c_{n}(t)n_{n}(t)\mathrm{d}t + \int_{t\in \mathbb{R}}c_{a}(t)n_{a}(t)\mathrm{d}t = N_{n}c_{n}^{*}(N_{n},N_{a}) + N_{a}c_{a}^{*}(N_{n},N_{a}).
\end{align}
Combining this expression with the operator's profit~\eqref{Eq:Profit}, the number of SAV commuters achieving the second-best optimum, $N^{SB*}$, is obtained as the solution to the following optimization problem:
\begin{align}
	\min_{N_{a}}\quad &SC^{SB}(N_{a}):=N_{n}c_{n}^{*}(N_{n},N_{a}) + N_{a}c_{a}^{*}(N_{n},N_{a}) - \pi\label{Eq:Second-Best}\\
	\text{s.t.}\quad 
	&0\leq N_{a} \leq N.\label{Eq:Second-Best_FlowConservation}
\end{align}
Condition~\eqref{Eq:Second-Best_FlowConservation} ensures that the number of NV commuters becomes non-negative.

The second-best fare $p^{SB*}$ is chosen to achieve the optimal number of SAV commuters.
It is obtained from the inverse demand function, which is derived by solving \eqref{Eq:DemandFunction} for $p$:
\begin{align}
	p(N_{a}) = m + AN - B - AN_{a}.\label{Eq:InvDemandFunction}
\end{align}
This inverse demand function gives the fare $p$ that satisfies the mode-choice indifference condition $c_{n}^{*} = c_{a}^{*}$ for a given $N_a$.
Substituting $N_a=N_a^{SB*}$ into \eqref{Eq:InvDemandFunction} yields $p^{SB*}$.\footnote{If the number of SAV commuters lies on a boundary (i.e., $N_{a} = 0$ or $N_{a} = N$), the fare can be higher or lower than Eq.~\eqref{Eq:InvDemandFunction} because cost equality is no longer required, as shown in \eqref{Eq:IntegratedEqui-Cond1}.
To avoid unnecessary complications, we nevertheless use the price given by \eqref{Eq:InvDemandFunction} in all cases.}

\subsection{SAV adoption and fare under the second-best policy}\label{Sec:SecondBest_Closed-form}
The number of SAV commuters in the equilibrium under the second-best policy (hereinafter, referred to as the second-best equilibrium) can be obtained from the first-order condition of the objective function.
The derivative with respect to $N_{a}$ is 
\begin{align}
	\cfrac{\partial SC^{SB}(N_{a})}{\partial N_{a}}
	=- \cfrac{\beta \gamma}{\mu(\beta + \gamma)}[(1-\kappa) + (1-\theta)]N + \cfrac{2\beta \gamma}{\mu(\beta + \gamma)}(1-\theta)N_{a} + B.\label{Eq:DerivativeSecondBestObj}
\end{align}
Setting \eqref{Eq:DerivativeSecondBestObj} to zero yields $N_a^{SB*}$ and the corresponding fare $p^{SB*}$, as stated in the following proposition:
\begin{prop}
	In the second-best equilibrium, the number of SAV commuters and the corresponding fare are
	\begin{align}
		&N_{a}^{SB*} =\cfrac{AN(1+\eta) - B}{2A},\label{Eq:Second-best-SAVNum}\\
		&p^{SB*} = m + \cfrac{(1-\eta)AN - B}{2}.\label{Eq:Second-best-fare}
	\end{align}
	If $N_{a}^{SB*}$ is below $0$ or above $N$, the number of SAV commuters is given by the corresponding boundary value, $0$ or $N$.
\end{prop}

The proposition implies that $N_a^{SB*}$ and $p^{SB*}$ depend on the capacity-VOT index $\eta$.
To see how the second-best outcomes differ from the equilibria under the three fare-setting scenarios, we compare $N_a^{SB*}$ with the corresponding equilibrium values of $N_{a}$ under each scenario.
This reveals that the ordering of these values depends on whether $\eta \geq 1$ or not.
Specifically, 
\begin{itemize}
	\item $\eta \geq 1$ $(\kappa \leq \theta)$: The equilibrium values satisfies the following ordering:
		\begin{align}
			N_{a}^{SB*} > N_{a}^{MC*} > N_{a2}^{AC*}\geq N_{a}^{m*}>N_{a0}^{AC*}=0.
		\end{align} 
		
	\item $\eta < 1$: The ordering is not predetermined; it depends on the underlying parameters.
\end{itemize}

\noindent This implies that, unlike the comparison of equilibrium commuting costs, MC pricing is not always the best policy for reducing the social cost.
When $\eta \geq 1$, we have $1-\eta \leq 0$. 
Eq.~\eqref{Eq:Second-best-fare} shows that $p^{SB*}$ is below the marginal cost $m$, and $N_{a}^{SB*}$ exceeds $N_{a}^{MC*}$.
In this case, MC pricing promotes SAV adoption toward the second-best level and yields a lower social cost than the other fare-setting scenarios.
By contrast, when $\eta < 1$, we obtain $1-\eta > 0$; hence, $p^{SB*}$ may exceed $m$.
This implies that MC pricing may induce excessive SAV adoption which leads to a higher social cost than that under the other fare-setting scenarios.

%
%
%
%

To clarify the reason why the SAV adoption becomes excessive, we rewrite the derivative of the social-cost function in terms of commuters' costs and the operator's profit under the equilibrium:
\begin{align}
	\cfrac{\partial SC^{SB}(N_{a}^{*})}{\partial N_{a}^{*}}
	= 
	\underbrace{- \cfrac{(1-\kappa)\beta \gamma}{\mu(\beta + \gamma)}}_{\partial c_{n}^{*}/\partial N_{a}^{*}}N_{n}^{*} 
	+ \underbrace{\left[\cfrac{(\kappa-\theta)\beta \gamma}{\mu(\beta + \gamma)} + \cfrac{\partial p(N_{a}^{*})}{\partial N_{a}^{*}}\right]}_{\partial c_{a}^{*}/\partial N_{a}^{*}}N_{a}^{*}
	- \underbrace{\left[(p(N_{a}^{*})-m) + \cfrac{\partial p(N_{a}^{*})}{\partial N_{a}^{*}}N_{a}^{*}\right]}_{\partial \pi/\partial N_{a}^{*}}.
\end{align}

\noindent The first term is the marginal change in commuting cost incurred by NV commuters.
The second and third terms describe the change in the total cost of SAV commuters. 
The second term captures how additional SAV commuters affect the congestion externalities they face, whereas the third term captures the change in the fare they pay, given that the mode choice equilibrium is preserved.
The fourth and fifth terms describe the change in the operator's profit. 
The third and fifth terms offset each other because both represent income transfers effected through the fare.

This equation reveals the following mechanism: when $\eta < 1$, adding more SAV commuters increases the congestion externalities they themselves face; consequently, if SAV adoption exceeds its efficient level, the social cost can rise.
The second term shows that the direction of the congestion-externality change for SAV commuters depends on the sign of $\kappa - \theta$.
When $\kappa > \theta$ (equivalently $\eta < 1$), this term becomes positive, indicating that an increase in SAV commuters raises their total commuting cost: because the capacity effect is too weak to offset the additional congestion, the extra SAV flow ultimately raises the total commuting cost.
As $N_{a}$ becomes large, the additional commuting cost incurred by the expanding SAV cohort can outweigh the cost savings achieved by NV commuters, thereby increasing the social cost.\footnote{This analytical finding is consistent with the results of \cite{Van_den_Berg2016-mv}.
In that study, the VOT effect is likewise denoted by $\theta$, while the capacity effect is expressed as the function $r[f]$ of the SAV market share $f$. 
Setting $\partial r[f]/\partial f = 0$ makes their specification equivalent to the constant capacity factor $\kappa$ used here.
Then, the authors state:
\begin{quote}
	When $\theta > r[f] + f\cdot \partial r/\partial f$, raising $f$ decreases the travel cost in an autonomous car. The capacity effect ..., making an increased share more beneficial for all users. 
	A smaller $\theta$ means that ... this user imposes longer travel times on those who arrive closer to $t^{*}$. 
	Hence, a smaller $\theta$ strengthens the heterogeneity effect, making increasing $f$ less beneficial for current autonomous car users.
\end{quote}
Our result is consistent with this mechanism.}

It is worth emphasizing that the additional externalities created when commuters shift from NVs to SAVs are converted into a reduction in the operator's profit.
In the second-best equilibrium, the fare is adjusted so that the mode choice condition $c_{n}^{*} = c_{a}^{*}$ continues to hold.
As a result, the increase in the congestion externalities for SAV commuters (second term) is offset by a corresponding decrease in the SAV fare (third term); this decrease causes the corresponding increase in the fifth term, which appears as a decrease in the operator's revenue.
Hence, although the increase in the social cost appears to stem from a decline in the operator's revenue, attempting to restore that profit by further promoting SAV use would increase social cost even more.
Therefore, reducing the social cost requires an appropriate fare regulation scheme that resolves the underlying problem of excessive SAV usage.

\subsection{Social cost comparison under different fare-setting scenarios}
In the previous section, we showed that when $\eta < 1$, the ordering of the number of SAV commuters in the second-best equilibrium relative to the equilibria under the three other fare-setting scenarios is not predetermined. 
This implies that which of those scenarios yields the lowest social cost is likewise non-obvious and depends on the underlying parameters. 
Building on this insight, we now compare social costs across the three different fare-setting scenarios.

Throughout this section, we assume that Eq.~\eqref{Eq:Cond_NonUnique} holds strictly, so that SAVs are used under both AC pricing and unregulated monopoly pricing.
Specifically, we impose the following condition on the population size $N$:
\begin{align}
	N > N_{\mathrm{min}}:=\cfrac{B+ \sqrt{4AF_{a}}}{A}.
\end{align}

Substituting the equilibrium commuting costs and fares under each scenario, we obtain the corresponding social costs as follows:
\begin{itemize}
	\item MC pricing:
		\begin{align}
			SC^{MC*} = \left[\cfrac{\kappa AN + (1-\kappa)B}{1-\theta} + t_{f} + F_{n}\right]N + F_{a}.\label{Eq:SC_Marginal}
		\end{align}
		
	\item AC pricing: let $SC^{AC*}_{0}$ and $SC^{AC*}_{2}$ denote the social costs when the numbers of SAV commuters are $N^{AC*}_{0}$ and $N^{AC*}_{2}$, respectively.
		For simplicity, we omit the unstable equilibrium.
		Then,
		\begin{align}
			&SC^{AC*}_{0} = \left[\cfrac{A}{1-\theta}N + t_{f} + F_{n}\right]N + F_{a},\\
			&SC_{2}^{AC*} = \left[\cfrac{(1+\kappa)AN+(1-\kappa)B - (1-\kappa)\sqrt{(AN-B)^{2}-4AF_{a}}}{2(1-\theta)}  + t_{f} + F_{n}\right]N.
		\end{align}
		Note that when the number of SAV commuters is $N^{AC*}_{0} = 0$, no SAV commuters are present, and the operator therefore cannot recover the fixed cost.
		
	\item Natural monopoly:
		\begin{align}
			SC^{m*} &=  \left[\cfrac{(1+\kappa)AN + (1-\kappa)B}{2(1-\theta)} + t_{f} + F_{n}\right]N -\left[ \cfrac{(AN - B)^{2}}{4A} - F_{a}\right].
		\end{align}
\end{itemize}

By comparing these expressions, we find that the fare-setting that minimizes social cost depends primarily on the value of $\eta$.\footnote{The formal proofs of the following propositions in this section are provided in \Cref{Sec:App-SC_Comparison}.}
We first obtain the following proposition:
\begin{prop}
	Assume $\eta \geq 1$.
	Then, the following relation holds:
	\begin{align}
		SC^{MC*} < SC^{AC*}_{2} < SC^{m*} < SC^{AC*}_{0}.
	\end{align}	
\end{prop}

\noindent This proposition implies that, when $\eta$ is sufficiently large, MC pricing always yields the highest welfare. 
The reason is that for $\eta \geq 1$, an increase in the number of SAV commuters reduces commuting costs for all commuters, as explained in the preceding section. 
Hence, encouraging SAV adoption through MC pricing is socially desirable in this case.

We next derive the following proposition concerning the social costs when $\eta < 1$:
\begin{prop}
	Assume $1 > \eta \geq 1/2$.
	Then,
	\begin{align}
		\begin{cases}
			SC^{MC*} < SC^{AC*}_{2} < SC^{m*} < SC^{AC*}_{0}\quad &\text{if}\quad N_{\min} < N < N^{MC=AC}_{c},\\
			SC^{MC*} = SC^{AC*}_{2} < SC^{m*} < SC^{AC*}_{0}\quad &\text{if}\quad N = N^{MC=AC}_{c},\\
			SC^{AC*}_{2} < SC^{MC*} < SC^{m*} < SC^{AC*}_{0}\quad &\text{if}\quad N^{MC=AC}_{c} < N,
		\end{cases}
	\end{align}
	\textcolor{black}{where $N^{MC=AC}_{c}$ is the population threshold at which the ordering of social costs under MC and AC pricing reverses.}\footnote{\textcolor{black}{The closed-form expressions of the population thresholds used throughout this section are also provided in \Cref{Sec:App-SC_Comparison}.}}
\end{prop}
\noindent This proposition shows that when $\eta$ is small, AC pricing can be more efficient than MC pricing at high population levels. 
As $N$ grows, the SAV share implied by MC pricing becomes excessive relative to the second-best SAV share; the higher fare under AC pricing then works like a Pigouvian tax on the congestion externalities generated by SAVs and can achieve a lower social cost.

We further establish the proposition that characterizes social cost when $\eta$ falls below one-half:
\begin{prop}
	Assume $1/2 > \eta > 0$.
	Define the critical fixed-cost level
	\begin{align}
		F_{a,c}:=\cfrac{\eta^{2}B^{2}}{(1-2\eta)^{2}A}.
	\end{align}
	Whether the critical population level $N^{MC=AC}_{c}$ exceeds the minimum feasible population $N_{\mathrm{min}}$ depends on whether the fixed cost $F_{a}$ is below $F_{a,c}$.

%
%

	Assume $F_{a} \geq F_{a,c}$.
	Then,
	\begin{align}
		\begin{cases}
			SC^{AC*} < SC^{m*} < SC^{MC*}< SC^{AC*}_{0}\quad &\text{if}\quad N_{\mathrm{min}} < N < N^{AC=m}_{c},\\
			SC^{m*} \leq SC^{AC*} < SC^{MC*}< SC^{AC*}_{0}\quad &\text{if}\quad N^{AC=m}_{c} \leq N,
		\end{cases}
	\end{align}
	\textcolor{black}{where $N_{c}^{AC=m}$ is the population threshold at which the ordering of social costs under AC and unregulated monopoly pricing reverses.}
	

	Assume $F_{a} < F_{a,c}$.
	Then, 
	\begin{align}
		\begin{cases}
			SC^{MC*} < SC^{AC*} < SC^{m*} < SC^{AC*}_{0}\quad &\text{if}\quad N_{\mathrm{min}} < N < N^{MC=AC}_{c}\\
			SC^{AC*} \leq SC^{MC*} < SC^{m*} < SC^{AC*}_{0}\quad &\text{if}\quad N^{MC=AC}_{c} \leq N < N^{MC=m}_{c}\\
			SC^{AC*} < SC^{m*} \leq SC^{MC*} < SC^{AC*}_{0}\quad &\text{if}\quad N^{MC=m}_{c} \leq N < N^{AC=m}_{c}\\
			SC^{m*} \leq SC^{AC*} < SC^{MC*} < SC^{AC*}_{0}\quad &\text{if}\quad N^{AC=m}_{c} \leq N
		\end{cases}
	\end{align}
	\textcolor{black}{where $N_{c}^{MC=m}$ is the population threshold at which the ordering of social costs under MC and unregulated monopoly pricing reverses.}
	
\end{prop}

\noindent This proposition shows that, for sufficiently small $\eta$, the natural-monopoly scenario can yield the lowest social cost among the three fare-setting scenarios.
Indeed, as $\eta \rightarrow 0$ (i.e., $\kappa \rightarrow 1$), the following qualitative result holds.
\begin{coro}
	As $\eta \rightarrow 0$, the number of SAV commuters and the fare in the second-best equilibrium approach those under unregulated monopoly pricing; that is,
	\begin{align}
		N_{a}^{SB*}\rightarrow \cfrac{AN-B}{2A},\quad p^{SB*} \rightarrow m + \cfrac{AN-B}{2}.
	\end{align}	
\end{coro}
\noindent Hence, when the capacity effect of SAVs is negligible, leaving the service unregulated is the welfare-maximizing policy.

The natural monopoly becomes welfare-optimal because, as discussed in \Cref{Sec:Equilibrium}, shifting commuters from NVs to SAVs no longer affects commuting costs when $\kappa$ approaches $1$.
Formally, the sensitivity of social cost gives
\begin{align}
	\cfrac{\partial SC^{SB}(N_{a}^{*})}{\partial N_{a}^{*}} 
	&= - (1-\kappa)\cfrac{\beta \gamma}{\mu(\beta + \gamma)} N - \cfrac{\partial \pi}{\partial N_{a}^{*}}
	\rightarrow - \cfrac{\partial \pi}{\partial N_{a}^{*}}.
\end{align}
As the equation shows, the first term, which represents the change in total commuting cost, vanishes when $\kappa = 1$. 
In that case, the change in social cost equals the change in profit, and the second-best equilibrium therefore coincides with the profit-maximizing outcome.


\begin{figure}[t]
	\centering
	\hspace{0mm}
    \includegraphics[width=0.8\linewidth,clip]{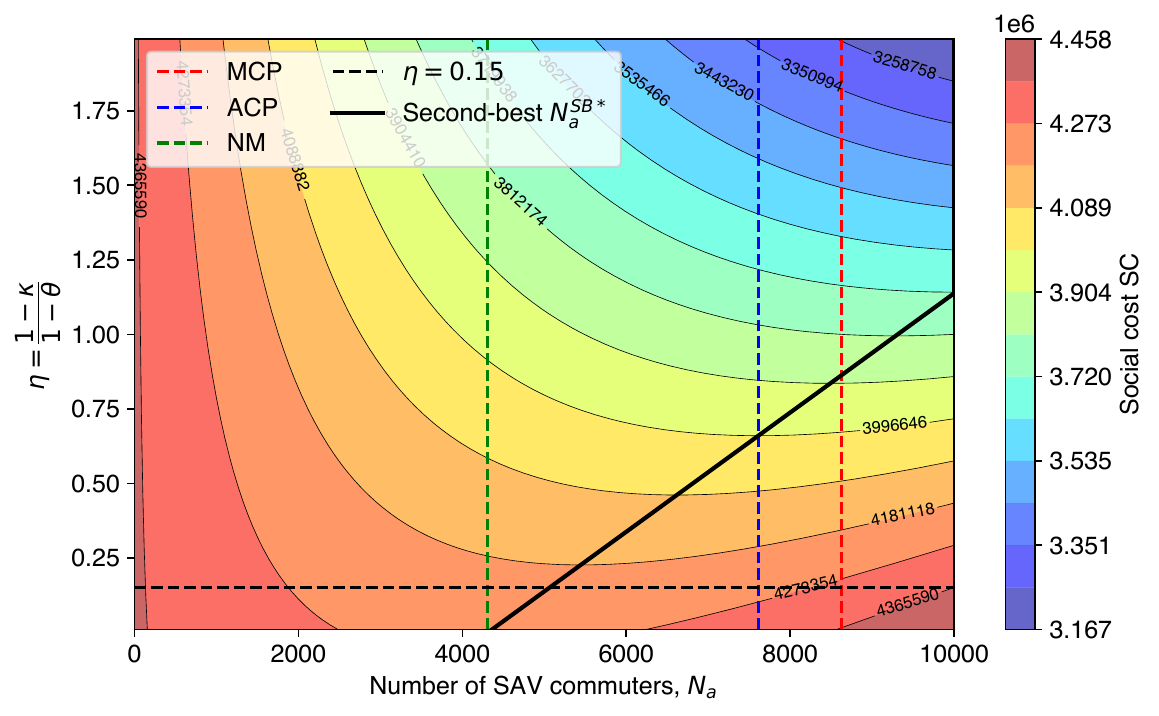}
    \vspace{-2mm}
    \caption{Contour plot of the social cost. \textcolor{black}{$N=10000$, $\theta = 0.5$, $\beta = 0.3$, $\gamma = 2.0$, $\mu = 20$, $t_{f} = 2$, $F_{n} = 300$, $m=210$, $w=100$, $F_{a} = 50000$.}}
    \vspace{-0mm}
    \label{Fig:Contour}
\end{figure}

\Cref{Fig:Contour} shows a contour plot of the social cost as a function of the number of SAV commuters, $N_{a}$, and the index $\eta$. In the plot, $\theta$ is held constant while $\kappa$ varies, so changes in $\eta$ arise solely from changes in $\kappa$.
First, when $\eta$ takes small values, the contours protrude toward the lower-left corner. 
This shape reveals that increasing $N_{a}$ does not necessarily reduce the social cost. 
For example, along the line $\eta = 0.15$, the social cost initially decreases as $N_{a}$ increases, but begins to increase once the curve bends.
As $\eta$ approaches zero, the socially optimal $N_{a}$ converges to the monopoly level $N_{a}^{m*}$.
By contrast, when $\eta$ becomes large, the contours no longer protrude to the lower left. 
In this region, the capacity-expansion effect is strong, and increasing $N_{a}$ is socially desirable. 
Above a certain threshold for $\eta$, it is optimal for all commuters to use SAVs, as indicated by the closed-form solution for $N_{a}^{*}$.

\subsection{Pricing strategy for social-cost reduction and Pareto improvement}
Based on the social cost analysis, we finally link the social cost to the pricing strategy introduced in \Cref{Sec:Equilibrium}.
We focus on the relationship between equilibrium commuting costs and social costs under AC pricing and unregulated monopoly pricing.
If $\eta \geq 1$ or $N < N_{c}^{AC=m}$, AC pricing dominates unregulated monopoly pricing in terms of both commuters' equilibrium cost and social cost:
\begin{align}
c^{AC*}_{2} < c^{m*},\quad \text{and}\quad
SC^{AC*}_{2} < SC^{m*}.
\end{align}
Hence, the pricing strategy can improve both cost measures.

Meanwhile, when $0 < \eta < 1$ and the population size exceeds the critical level $N_{c}^{AC=m}$, the welfare ordering reverses:
\begin{align}
	c^{AC*}_{2} < c^{m*},\quad \text{and}\quad
	SC^{AC*}_{2} > SC^{m*},
\end{align}
Promoting SAV adoption through AC pricing still reduces commuting costs, yet it increases the social cost. The reason is that a small $\eta$, reflecting a weak capacity effect $\kappa$ with a strong VOT effect $\theta$, provides only limited social gains from further SAV adoption. Switching to AC pricing increases the number of SAV commuters (see \Cref{Sec:Equilibrium_Compare}). However, because the capacity effect is weak, this increase worsens congestion among SAV commuters, as discussed in \Cref{Sec:SecondBest_Closed-form}. While the negative impact on commuters' costs is offset by the fare reduction under AC pricing, this fare reduction, in turn, lowers the service provider's profit. As a result, the social cost can increase even when commuting costs decrease.

The discussion suggests that a pricing strategy that switches between fare-setting rules based solely on SAV adoption can increase the social cost.
If a policymaker's concern is confined to commuters' costs, such a strategy is sufficient, as it reduces them.
By contrast, if the policymaker also aims to reduce the social cost, this threshold rule can be augmented with a second criterion: the technological maturity of the capacity effect, captured by $\eta$.
Specifically, we propose the following pricing strategy:
\begin{description}
\item[\textbf{Step 1}] \textit{No fare regulation.}
Keep the fare unregulated and allow monopoly pricing.

\item[\textbf{Step 2}] \textit{Technological maturation under monopoly.}
Maintain monopoly pricing until $\eta\geq 1$ holds.
	Operationally, wait until $\kappa$ has fallen sufficiently.

\item[\textbf{Step 3}] \textit{Introduce AC fare regulation.}
Once $\eta\geq 1$ is satisfied, implement AC pricing.
\end{description}

\noindent Under this strategy, the number of SAV commuters rises to $N_{a2}^{AC*}$.
In moving from the monopoly equilibrium to the AC-pricing equilibrium, both the equilibrium commuting cost and the social cost reduce; a Pareto improvement is achieved.

%
%
%

This pricing strategy suggests a simple rule: postpone fare regulation until the capacity-expansion technology has matured, for example, until SAV platooning can reliably sustain shorter headways.
In the early diffusion stage, the capacity effect would be weak ($\eta$ is low), so allowing unregulated monopoly pricing is socially preferable. 
During this period, resources should be channelled toward accelerating technological advances that strengthen capacity expansion. 
Once the technology matures and SAV adoption expands, switching to AC pricing becomes socially beneficial for SAV commuters.


\begin{figure}[t]
	\centering
	\hspace{0mm}
    \includegraphics[width=0.8\linewidth,clip]{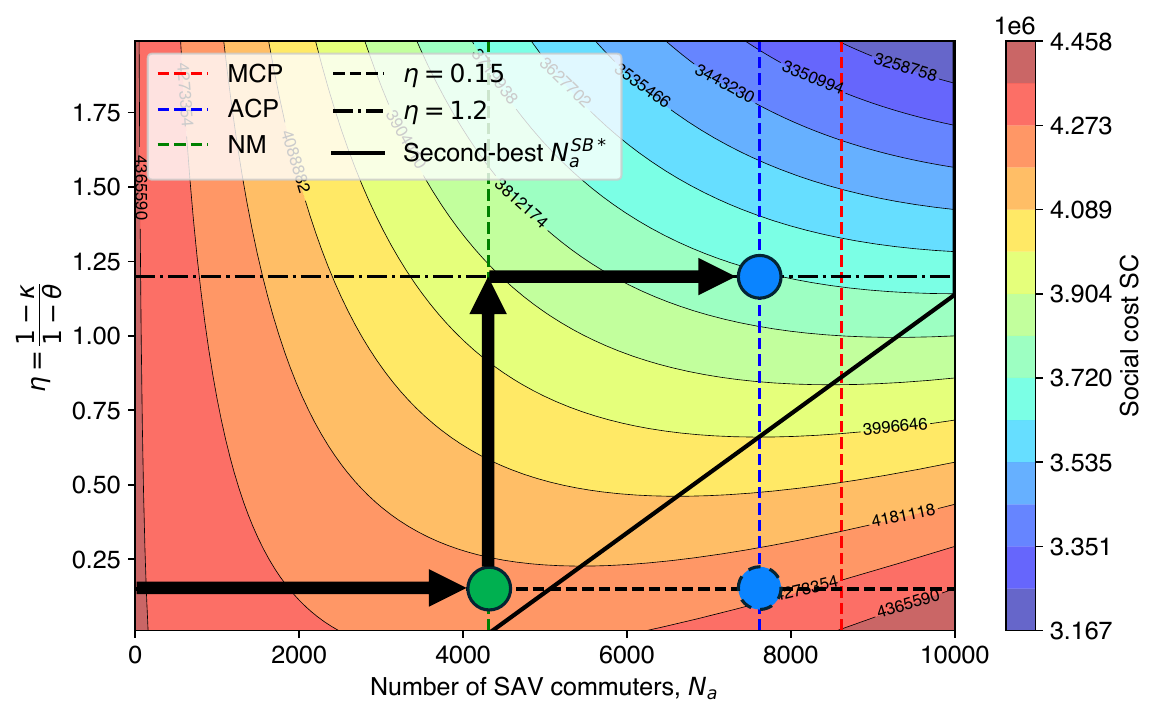}
    \vspace{-2mm}
    \caption{Regulatory strategy combined with the index $\eta$}
    \vspace{-0mm}
    \label{Fig:RegulatoryStrategy}
\end{figure}

\Cref{Fig:RegulatoryStrategy} shows the same social-cost contour plot as in \Cref{Fig:Contour}, now used to illustrate the effectiveness of the proposed pricing strategy.
Assume first that $\eta = 0.15$.
Under unregulated monopoly pricing, the number of SAV commuters increases to the monopoly-equilibrium level (the green circle in the figure). 
Switching to AC pricing increases the number of SAV commuters further (the blue circle with a dashed outline); however, it raises the social cost.
Next, we follow Phase 2 of the proposed strategy and maintain unregulated monopoly pricing until $\eta$ becomes $1.2$.
The figure indicates that the traffic state has shifted into a region where the capacity-expansion effect is strong enough that additional SAV commuters reduce the social cost.
At this point, switching to AC pricing (Phase 3) further increases the number of SAV commuters while simultaneously reducing the social cost (shifting to the blue circle).

In summary, our results identify $\eta$ as the key parameter that aligns reductions in social cost with reductions in equilibrium cost in both first-best and second-best policy design.
This finding underscores the importance of allocating limited innovation budgets to capacity-expansion technologies that deliver social benefits.
Therefore, effective utilization of SAVs in the transport system requires fare regulation that is tied not only to the current level of SAV adoption but also to the maturity of the underlying technology; combining these two triggers is essential for realizing the full social benefit of SAVs.

\section{Conclusion}\label{Sec:Conclusion}
This paper examines how scale economies in SAV operations affect the efficiency of a mixed-traffic transportation system by analyzing three fare-setting scenarios: MC pricing, AC pricing, and unregulated monopoly pricing.
First, our equilibrium analysis showed that MC pricing, although it forces the operator to run at a deficit, yields the lowest commuting cost among the three scenarios.
We also demonstrated a two-step pricing strategy: allow monopoly pricing initially and switch to AC pricing once SAV adoption has passed a critical threshold.
This strategy eliminates the deficit while promoting SAV adoption and reducing commuting costs.
In addition, we find that the Downs--Thomson paradox can arise in SAV system under AC pricing: capacity expansion can increase commuting costs because the SAV adoption decreases.
These findings confirm the superiority of MC pricing for reducing commuting costs, and demonstrate that a well-timed implementation of AC pricing can achieve a convenient and economically sustainable transportation system for commuters.

Next, we evaluate the social cost, defined as the commuters' total costs minus the operator's profit.
When a time-varying congestion toll can be implemented alongside fare regulation, the first-best policy is MC pricing combined with the optimal toll.
Using the closed-form solution, we derive the condition under which implementing the optimal toll to the MC equilibrium produces a Pareto improvement: the SAV capacity effect must exceed its VOT effect.
We also propose that the toll revenue is used for financing the road capacity and show that the self-financing principle holds.
We then solve the second-best problem, in which only the fare can be regulated, and compare its solution with the three fare-setting scenarios.
The results show that MC pricing does not always minimize the social cost.
When the capacity effect is weaker than the VOT effect, promoting SAV adoption can increase congestion externalities; their impact is reflected in a decrease in the operator's profit.
In the extreme case of a negligible capacity effect, unregulated monopoly pricing becomes the most efficient.
Overall, the impact of fare regulation can become far from intuitive.
The choice and timing of MC or AC pricing must be tailored to the current SAV adoption level and the maturity of the underlying technology to achieve a socially efficient transportation system.

A natural direction for future research is to extend our model spatially, for example, to a corridor or network with multiple bottlenecks. In such a setting, it is important to examine the spatio-temporal patterns of NV and SAV use and to assess how fare regulation alters those patterns.
\textcolor{black}{It would also be important to examine flow and cost patterns in long-run settings where travel demand may be elastic, for example, when commuters can adjust their residential locations and when the number of commuters is endogenously determined.} The literature shows that commuters sort themselves spatially according to their heterogeneity~\citep[e.g.,][]{Takayama2017-ge,Osawa2018-ip,Sakai2024-zq}. Hence, adding a mode-choice dimension is expected to yield a clear sorting pattern in which each mode is preferred by commuters living at specific distances from the common destination. \textcolor{black}{In addition, allowing total demand $N$ to respond to generalized travel cost~\citep[e.g.,][]{Arnott1993-jd,Fosgerau2018-ju} would be important for capturing potential induced-demand effects when SAVs lower travel costs. These extensions could then provide policy insights not only for fare regulation but also for land-use policy.
Incorporating an explicit parking model that accounts for parking location choice and capacity constraints~\citep[][]{Arnott2010-ei,Liu2018-pa,Tian2019-ug} is another promising direction. This would link NV parking costs to traffic conditions and provide a more refined assessment of NV--SAV adoption and the welfare impacts of fare regulation.}


\textcolor{black}{On the supply side, an important direction for future research is to explicitly incorporate SAV operational features, for example, by modeling pickup waiting time and deadheading as outcomes of an SAV operator's fleet decisions and operations. Such operational frictions can worsen congestion and, in some cases, may reduce effective bottleneck capacity as SAV usage increases (i.e., $\kappa > 1$). Because these supply-side mechanisms may interact with the scale economies considered in our framework, it would be important to examine whether, and to what extent, the equilibrium outcomes and policy implications identified here are preserved under such extensions. Explicitly modeling the operator's fleet decisions is also closely related to a key advantage of SAVs: vehicles can be redeployed for other tasks after drop-off. It would therefore be interesting to incorporate this feature explicitly and to assess the impacts of SAV adoption from this perspective.}
Another important direction is to incorporate public transit. Because transit operations would be subject to different scale economies, their inclusion is likely to generate additional multiple equilibria with varying shares of NV, SAV, and transit commuters. Analyzing them would clarify how best to combine SAV services with conventional transit.
It would also be interesting to explore new schemes for redistributing the service provider's profit. This study shows that the second-best policy can coincide with profit maximization by a monopolistic provider; accordingly, an appropriate redistribution of those profits is crucial for broad social acceptance of the policy.
Finally, it would be worthwhile to explore alternative market structures, such as oligopolistic competition among SAV operators. These analyses deepen our understanding of SAV systems that incorporate scale economies.

\section*{Acknowledgement}
We thank Minoru Osawa for their helpful comments. 
We are also grateful to the associate editor and the anonymous referees for their valuable comments.
This work was supported by JSPS KAKENHI (Grants JP23K13418, JP25K01339, and JP25H00751), JST FOREST Program (Grant JPMJFR215M), and the Council for Science, Technology and Innovation (Grant JPJ012187).

\appendix

\section{Closed-form solution of mode-choice equilibrium}\label{Sec:App-ClosedFormSolution_DUE}
Using the mode-specific commuting costs~\eqref{Eq:DepartEquiCost-normal} and \eqref{Eq:DepartEquiCost-autonomous}, the equilibrium condition $c^{*}_{n} = c_{a}^{*}$ can be written as
\begin{align}
	\cfrac{\beta \gamma}{\mu(\beta + \gamma)}(N_{n} + \kappa N_{a}) + t_{f} + F_{n} = \cfrac{\beta \gamma}{\mu(\beta + \gamma)}(\theta N_{n} + \kappa N_{a}) + \theta t_{f} + p  +  w.
\end{align}
Applying the flow-conservation condition $N_{n} + N_{a}$, this expression reduces to
\begin{align}
	\cfrac{\beta \gamma (1-\theta)}{\mu(\beta + \gamma)} (N - N_{a}) - (p-m)
	= \theta t_{f}+ m +  w - t_{f}  - F_{n}.\label{Eq:App1-BaseEquation}
\end{align}
Below, we solve this equation for $N_{a}$ under each fare-setting scenario and then derive the corresponding equilibrium commuting cost.

\subsection{MC pricing}
Assume $p=m$.
Substituting this into Eq.~\eqref{Eq:App1-BaseEquation}, we obtain the equilibrium number of SAV commuters:
\begin{align}
	N_{a}^{MC*} = N - \cfrac{\mu (\beta + \gamma)}{\beta \gamma (1-\theta)}(\theta t_{f} + m +  w - t_{f} - F_{n})
	=\cfrac{AN - B}{A}.
\end{align}
Accordingly,
\begin{align}
	N_{n}^{MC*} = N - N_{a}^{MC*} = \cfrac{B}{A}.
\end{align}
Thus, when $AN-B>0$ and $B>0$, $N_{n}^{MC*} > 0$ and $N_{a}^{MC*} > 0$.
In this case, the equilibrium commuting cost is 
\begin{align}
	c^{MC*} = \cfrac{\kappa AN + (1-\kappa)B}{1-\theta} + t_{f} + F_{n}.
\end{align}

If $B \leq 0$, the equilibrium condition $c^{*}_{n} = c_{a}^{*}$ implies $N_{n}^{MC*}\leq 0$.
The solution therefore lies on the boundary, so $N_{n}^{MC*} = 0$ and $N_{a}^{MC*} = N$.
Conversely, when $AN-B\leq 0$, $N_{n}^{MC*} = N$ and $N_{a}^{MC*} = 0$.

\subsection{AC pricing}
Assume $p = m + F_{a}/N_{a}$.
First, if $N_{a} = 0$, the fare $p$ diverges to infinity, so $c_{a}^{*}$ also becomes unbounded.
Consequently, the equilibrium becomes the boundary solution $N_{n0}^{AC*} = N$ and $N_{a0}^{AC*} = 0$.

Consider the case where $N_{a} > 0$.
Substituting the equation into Eq.~\eqref{Eq:App1-BaseEquation} yields the quadratic equation,
\begin{align}
	A(N_{a})^{2} - (AN-B)N_{a} + F_{a} = 0.
\end{align}
When the discriminant $(AN-B)^{2} - 4AF_{a} \geq 0$, the equation has two candidate solutions:
\begin{align}
	N_{a1}^{AC*}=\cfrac{AN-B-K}{2A},\quad N_{a2}^{AC*}=\cfrac{AN-B+K}{2A}.
\end{align}
where $K = \sqrt{(AN-B)^{2}-4AF_{a}}$.
The corresponding numbers of NV commuters are
\begin{align}
	N_{n1}^{AC*}=\cfrac{AN + B + K}{2A},\quad N_{n2}^{AC*}=\cfrac{AN + B - K}{2A}.
\end{align}
Because all of these expressions are strictly positive when $AN-B>0$ and $B>0$, these equilibrium solutions are valid and do not lie on either boundary.

The corresponding equilibrium commuting costs are as follows:
\begin{align}
	&c^{AC*}_{0} = \cfrac{A}{1-\theta}N + t_{f} + F_{n},\\
	&c^{AC*}_{1} = \cfrac{(1+\kappa)AN+(1-\kappa)B + (1-\kappa)K}{2(1-\theta)}  + t_{f} + F_{n},\\
	&c^{AC*}_{2} = \cfrac{(1+\kappa)AN+(1-\kappa)B - (1-\kappa)K}{2(1-\theta)}  + t_{f} + F_{n}.
\end{align}

\subsection{Unregulated monopoly pricing}
Substituting the optimal fare~\eqref{Eq:Fare_Monopoly} yields the following equation:
\begin{align}
	A (N - N_{a}) - \cfrac{AN-B}{2}
	- B = 0.
\end{align}
Solving this equation, we have 
\begin{align}
	N_{a}^{m*} = \cfrac{AN-B}{2A},\quad N_{n}^{m*} = \cfrac{AN + B}{2A}.
\end{align}
Because these values are positive when $AN-B>0$ and $B>0$, the equilibrium does not lie on either boundary.

The equilibrium commuting cost is
\begin{align}
	c^{m*} &= \cfrac{(1+\kappa)AN + (1-\kappa)B}{2(1-\theta)} + t_{f} + F_{n}.
\end{align}

\section{Stability analysis under AC pricing}\label{Sec:App-StabilityAnalysis}

\subsection{Formal definition of the local asymptotic stability}
Let $N_{a}^{*}$ denote the equilibrium number of SAV commuters.
The equilibrium is \textit{Lyapunov stable} if, for every $\epsilon > 0$ there exists $\delta > 0$ such that
\begin{align}
	|N_{a}(0) - N_{a}^{*}| < \delta  \;\;\Longrightarrow\;\; 
	|N_{a}(\tau) - N_{a}^{*}| < \epsilon\quad \forall \tau \geq 0.
\end{align}
The equilibrium is \textit{attractive} if there exists $\bar{\delta}>0$ such that
\begin{align}
	|N_{a}(0)-N_{a}^{*}|<\bar{\delta} 
	\;\;\Longrightarrow\;\;
	\lim_{\tau\to\infty} N_{a}(u)=N_{a}^{*}.
\end{align}
An equilibrium is \textit{locally asymptotically stable} if it satisfies both properties simultaneously.

We thus prove that $N_{a0}^{AC*}$ and $N_{a2}^{AC*}$ are Lyapunov stable and attractive, and $N_{a1}^{AC*}$ is not Lyapunov stable and attractive.

\subsection{Stability of $N_{a0}^{AC*}$}
We begin with the following lemma:
\begin{lemm}
	There exists a constant $\bar{\delta}>0$ such that
\begin{align}
	N_{a0}^{AC*} < N_{a} < N_{a0}^{AC*} + \bar{\delta}
	\quad\Longrightarrow\quad
	c_{a}^{*}-c_{n}^{*}>0 .
	\end{align}	
\end{lemm}
\begin{proof}
	Substituting $N_{a}=\bar{\delta}$ $(\because N_{a0}^{AC*} = 0)$, we obtain the following relation for $c_{a}^{*}-c_{n}^{*}$:
	\begin{align}
		c_{a}^{*}-c_{n}^{*} = A\bar{\delta} + \cfrac{F_{a}}{\bar{\delta}} - (AN-B).
	\end{align}
	Since $\,\bar{\delta}\to0$ implies $F_{a}/\bar{\delta}\to\infty$ monotonically, for $\bar{\delta}$ chosen sufficiently small,
	$c_{a}^{*}-c_{n}^{*}>0$ holds for every $N_{a}\in(0,\bar{\delta})$.
\end{proof}

Using the lemma, we first show Lyapunov stability.
Fix an arbitrary $\epsilon>0$ and define
\begin{align}
	\delta(\epsilon):=\min\{\epsilon,\bar{\delta}\},
	\qquad
	\mathcal{N}_{\delta(\epsilon)}
	:=\bigl\{N_{a}\mid 0<N_{a}<\delta(\epsilon)\bigr\}.
\end{align}

\noindent Assume that the initial value $N_{a}(0)$ (at $u=0$) lies in this neighbourhood.
By the lemma, $c_{a}^{*}-c_{n}^{*}>0$ in $\mathcal{N}_{\delta(\epsilon)}$.
Because the evolutionary dynamics $V(N_{a})$ satisfy PC property, we have $V(N_{a})<0$ in $\mathcal{N}_{\delta(\epsilon)}$.
Thus, $N_{a}(u)$ always moves in the decreasing direction and never leaves $\mathcal{N}_{\delta(\epsilon)}$ for any $u\ge0$.
By definition, $N_{a0}^{AC*}$ is therefore Lyapunov stable.

We next show that the equilibrium is attractive.
Because $V(N_{a}) < 0$ throughout the neighbourhood $\mathcal{N}_{\delta(\epsilon)}$, the trajectory $N_{a}(u)$ is strictly decreasing.
Since $0 < N_{a}(u)\leq N_{a}(0) < \delta(\epsilon)$, the state is also bounded below by $0$.
A strictly decreasing function that is bounded below must converge, so the limit
\begin{align}
	L:=\lim_{u\to\infty}N_{a}(u)
\end{align}
exists and satisfies $0\leq L \leq \delta(\epsilon)$.
Continuity of $V$ implies
\begin{align}
	\lim_{u \to \infty}\dot{N}_{a}(u) = V(L).
\end{align} 
If $V(L)\neq 0$, the derivative keeps a fixed sign for all sufficiently large $u$;
the state would therefore continue decreasing indefinitely and could not settle at a finite limit, contradicting the existence of $L$.
Hence $V(L) = 0$, and NS property implies that $L = N_{a0}^{AC*}$.

$N_{a0}^{AC*}$ is therefore attractive.
Combined with the Lyapunov stability, this completes the proof that the equilibrium is aymptotically stable.

%
%
%


\subsection{Stability of $N_{a2}^{AC*}$}
We first show the following lemma:
\begin{lemm}
	There exists $\bar\delta>0$ such that
	\begin{align}
		&
		N_{a2}^{AC*}<N_{a}<N_{a2}^{AC*}+\bar\delta
		\;\Longrightarrow\;
		c_{a}^{*}-c_{n}^{*}>0,\\
		&
		N_{a2}^{AC*} - \bar\delta<N_{a}<N_{a2}^{AC*}
		\;\Longrightarrow\;
		c_{a}^{*}-c_{n}^{*}<0.
	\end{align}
\end{lemm}
\begin{proof}
	Consider a sufficiently small $\delta$ satisfying $0 < \delta < \bar\delta$.
	Suppose first that $N_{a}$ is within the range $N_{a2}^{AC*}<N_{a}<N_{a2}^{AC*}+\bar\delta$: $N_{a} = N_{a2}^{AC*} + \delta$.
	Substituting this into $c_{a}^{*}-c_{n}^{*}$, we have the following relation:
	\begin{align}
		c_{a}^{*}-c_{n}^{*} = \cfrac{\delta (2AN_{a2}^{AC*} + A \delta - (AN-B))}{N_{a2}^{AC*}+\delta}
		= \cfrac{\delta (\sqrt{(AN-B)^{2}- 4AF_{a}} + A \delta)}{N_{a2}^{AC*}+\delta}
		 > 0.
	\end{align}

	Consider next that $N_{a}= N_{a2}^{AC*} - \delta$, i.e., $N_{a}$ is within the range $N_{a2}^{AC*}-\bar\delta<N_{a}<N_{a2}^{AC*}$.
	Substituting this into $c_{a}^{*}-c_{n}^{*}$, we have
	\begin{align}
		c_{a}^{*}-c_{n}^{*} = \cfrac{-\delta (2AN_{a2}^{AC*} - A \delta - (AN-B))}{N_{a2}^{AC*}+\bar\delta} 
		= \cfrac{-\delta (\sqrt{(AN-B)^{2}- 4AF_{a}} - A \delta)}{N_{a2}^{AC*}+\delta}.
	\end{align}
	This implies that $c_{a}^{*}-c_{n}^{*} < 0$ if $\sqrt{(AN-B)^{2}- 4AF_{a}} - A \delta > 0$.
	It is obvious that this condition holds for an arbitrary $\delta$ if $\bar{\delta}$ is chosen sufficiently small.
	The lemma is thus proved.
\end{proof}

We first show Lyapunov stability.
Fix an arbitrary $\epsilon>0$ and define
\begin{align}
	\delta(\varepsilon):=\min\{\varepsilon,\bar\delta\},\qquad
	\mathcal N_{\delta(\varepsilon)}
	:=\{N_{a}\mid |N_{a}-N_{a2}^{AC*}|<\delta(\varepsilon)\}.
\end{align}

Inside $\mathcal N_{\delta(\varepsilon)}$, the lemma and PC property suggest that
\begin{align}
	\begin{cases}
		N_{a}>N_{a2}^{AC*}\; &\Rightarrow\; c_{a}^{*}-c_{n}^{*} > 0 \;\Rightarrow\; V(N_{a})<0,\\[3pt]
		N_{a}<N_{a2}^{AC*}\; &\Rightarrow\; c_{a}^{*}-c_{n}^{*} < 0\;\Rightarrow\; V(N_{a})>0.
	\end{cases}
\end{align}
Hence, $N_{a}(u)$ always approaches $N_{a2}^{AC*}$ and never leaves $\mathcal{N}_{\delta(\epsilon)}$ for any $u \geq 0$ if the initial value lies in the neighbourhood.
By definition, $N_{a2}^{AC*}$ is therefore Lyapunov stable.

We next prove that the equilibrium is attractive.
Let an initial state $N_{a}(0)$ be chosen in the neighbourhood $\mathcal{N}_{\delta(\epsilon)}$.
Because the lemma implies $V(N_{a}) < 0$ when $N_{a} > N_{a2}^{AC*}$ and $V(N_{a}) > 0$ when $N_{a} < N_{a2}^{AC*}$ in $\mathcal{N}_{\delta(\epsilon)}$, the trajectory $N_{a}(u)$ is strictly decreasing on the right of the equilibrium and strictly increasing on its left.
In either case, $N_{a}(u)$ stays in $\mathcal{N}_{\delta(\epsilon)}$ and is bounded between the two constants $N_{a2}^{AC*} - \delta(\epsilon)$ and $N_{a2}^{AC*} + \delta(\epsilon)$.
A monotone and bounded function must converges, so the limit
\begin{align}
	L:=\lim_{u\to\infty}N_{a}(u)
\end{align}
exists.
The continuity yields
\begin{align}
	\lim_{u \to \infty}\dot{N}_{a}(u) = V(L).
\end{align}
If $L > N_{a2}^{AC*}$, then $V(L)$ would be negative and the derivative would remain negative; the state continues decreasing past $L$, contradicting the definition of the limit.
The symmetric contradiction arises if $L < N_{a2}^{AC*}$.
Hence, neither inequality is possible and we must have $L = N_{a2}^{AC*}$.

Because trajectories converge to $N_{a2}^{AC*}$, the equilibrium is attractive.
Combined with the Lyapunov stability, this completes the proof that the equilibrium is aymptotically stable.

\subsection{Instability of $N_{a1}^{AC*}$}
We first show the following lemma:
\begin{lemm}
	For a sufficiently small $\bar\delta>0$, the following relations hold:
	\begin{align}
		&
		N_{a1}^{AC*}<N_{a}<N_{a1}^{AC*}+\bar\delta
		\;\Longrightarrow\;
		c_{a}^{*}-c_{n}^{*} < 0,\label{Eq:App-Instability1}\\
		&
		N_{a1}^{AC*} - \bar\delta<N_{a}<N_{a1}^{AC*}
		\;\Longrightarrow\;
		c_{a}^{*}-c_{n}^{*} > 0.\label{Eq:App-Instability2}
	\end{align}
\end{lemm}
\begin{proof}
	Substituting $N_{a1}^{AC*} + \bar\delta$, we have
	\begin{align}
		c_{a}^{*}-c_{n}^{*} = \cfrac{\bar\delta (A \bar\delta - \sqrt{(AN-B)^{2}- 4AF_{a}} )}{N_{a1}^{AC*}+ \bar\delta},
	\end{align}
	which means that $c_{a}^{*}-c_{n}^{*} < 0$ if $\bar\delta < \sqrt{(AN-B)^{2}- 4AF_{a}}/A$.
	
	Meanwhile, substituting $N_{a1}^{AC*} - \bar\delta$, we have
	\begin{align}
		c_{a}^{*}-c_{n}^{*} = \cfrac{ \bar\delta (A \bar\delta + \sqrt{(AN-B)^{2}- 4AF_{a}} )}{N_{a1}^{AC*} - \bar\delta},
	\end{align}
	which means that $c_{a}^{*}-c_{n}^{*} < 0$ if $\bar\delta < N_{a1}^{AC*}$.
	
	Therefore, for $\bar\delta:=\min\{\sqrt{(AN-B)^{2}- 4AF_{a}}/A, N_{a1}^{AC*}\} $, the relations~\eqref{Eq:App-Instability1} and \eqref{Eq:App-Instability2} hold.
\end{proof}

This lemma obviously suggests that $N_{a1}^{AC*}$ is no longer Lyapunov stable.
Define the neighborhood
\begin{align}
	\mathcal N_{\bar\delta}
	:=\{N_{a}\mid |N_{a}-N_{a1}^{AC*}|<\bar\delta\}.
\end{align}
The lemma and PC property implies that $V(N_{a}) > 0$ when $N_{a} > N_{a1}^{AC*}$ and $V(N_{a}) < 0$ when $N_{a} < N_{a1}^{AC*}$ in $\mathcal N_{\bar\delta}$.
Therefore, $N_{a}(u)$ never approaches $N_{a1}^{AC*}$ once the traffic state deviates from the equilibrium.
Since the equilibrium is not Lyapunov stable, it is not asymptotically stable.
Note that this analysis is valid when $(AN-B)^{2} - 4AF_{a} = 0$: $N_{a1}^{AC*} = N_{a2}^{AC*}$ and the equilibrium becomes unstable.

\section{Temporal-sorting property under the first-best policy}\label{Sec:App-TemporalSortingFB}
We first obtain the following lemma.
\begin{lemm}\label{Lemm:SortingProperty_Base}
	Let
	\(
	  \mathcal T
	  :=\left\{t\in\mathbb R \,\middle|\, n_{n}(t)>0\ \text{and}\ n_{a}(t)>0\right\}.
	\)
	Then the Lebesgue measure of \(\mathcal T\) is zero; that is, there is no time interval of positive length in which NV and SAV flows are both strictly positive.
\end{lemm}
\begin{proof}
	Suppose that, for some \(\bar t\), \(n_{n}(\bar t)>0\) and \(n_{a}(\bar t)>0\).  
	Optimality conditions \eqref{Eq:FirstBest_Optimality1}--\eqref{Eq:FirstBest_Optimality2} give
	\[
		s(\bar t)+t_f+F_n+\tau(\bar t)=c^{FB*},\qquad
	  	s(\bar t)+\theta t_f+m  + w   +\kappa\tau(\bar t)=c^{FB*}.
	\]
	Eliminating \(\tau(\bar t)\) yields  
	\[
		(1-\kappa)\,s(\bar t)
		= (1-\kappa)c^{FB*}-\bigl[\theta t_f+m  + w -\kappa(t_f+F_n)\bigr],
	\]
	so \(s(\bar t)\) must equal a constant.  
	Because $s(\cdot)$ is continuous and strictly monotone on each of $(-\infty,0]$ and $[0,\infty)$, for any $\ell\in\mathbb R$ the level set $\{t\in\mathbb R\mid s(t)=\ell\}$ contains at most one point in each half-line, and hence at most two points in total. Therefore, it has Lebesgue measure zero, and $\mathcal T$ cannot contain an interval of positive measure.
\end{proof}

\noindent The lemma implies that, at almost every time at which there is positive flow at the bottleneck, only one of the two vehicle classes has positive flow.
Note that there are isolated points $\bar{t}\in \mathcal{T}$ at which both NV and SAV flows can be positive; namely, the following condition holds:
\begin{align}
	s(\bar t)+t_f+F_n+\tau(\bar t)=
	s(\bar t)+\theta t_f+m  + w  +\kappa\tau(\bar t) = c^{FB*}.
\end{align}
Because these measure-zero points do not affect the essential flow--cost pattern, we assume that $n_{a}(\bar{t})=0$ for any $\bar{t}$.

To determine which vehicle class is flowing at time $t\in\mathbb{R}$, we subtract the left-hand sides of conditions \eqref{Eq:FirstBest_Optimality1} and \eqref{Eq:FirstBest_Optimality2} and obtain
\begin{align}
	\left[ s(t) + \theta t_{f} + m  + w  + \kappa \tau(t)  \right] - \left[s(t) + t_{f} + F_{n} + \tau(t) \right]
	= 
	\begin{cases}
		>0,\quad &\text{if}\quad \tau(t) < \cfrac{B}{1-\kappa},\\
		0,\quad &\text{if}\quad \tau(t) = \cfrac{B}{1-\kappa},\\
		<0,\quad &\text{if}\quad \tau(t) > \cfrac{B}{1-\kappa}.
	\end{cases}
\end{align}

Combining this with the complementarity conditions~\eqref{Eq:FirstBest_Optimality1} and \eqref{Eq:FirstBest_Optimality2}, we have the following relationships:
\begin{itemize}
	\item If $n_{n}(t) > 0$, then
		\begin{align}
			\left[ s(t) + \theta t_{f} + m  + w  + \kappa \tau(t)  \right] - \left[s(t) + t_{f} + F_{n} + \tau(t) \right] \geq 0.
		\end{align}
		We thus have $\tau(t) \leq B/(1-\kappa)$.
		We also have $s(t) + t_{f} + F_{n} + \tau(t) = c^{FB*}$.
		Therefore,
		\begin{align}
			s(t) = c^{FB*} - (t_{f} + F_{n} + \tau(t)) \geq 
			c^{FB*} - (t_{f} + F_{n}) - \cfrac{B}{1-\kappa}.
		\end{align}
		
	\item If $n_{a}(t) > 0$, then
		\begin{align}
			\left[s(t) + t_{f} + F_{n} + \tau(t) \right] - \left[ s(t) + \theta t_{f} + m  + w  + \kappa \tau(t)  \right] > 0.
		\end{align}
		We thus have $\tau(t) > B/(1-\kappa)$.
		We also have $s(t) + \theta t_{f} + m  + w  + \kappa \tau(t) = c^{FB*}$.
		Therefore,
		\begin{align}
			s(t) = c^{FB*} - (B + t_{f} + F_{n}  + \kappa \tau(t)) < 
			c^{FB*} - (t_{f} + F_{n}) - \cfrac{B}{1-\kappa}.
		\end{align}
\end{itemize}
In summary, we have the following relationship:
\begin{align}
s(t) 
\begin{cases}
	\geq c^{FB*} - (t_{f} + F_{n}) - \cfrac{B}{1-\kappa}\quad &\text{if}\quad n_{n}(t) > 0,\\
	< c^{FB*} - (t_{f} + F_{n}) - \cfrac{B}{1-\kappa}\quad &\text{if}\quad n_{a}(t) > 0.
\end{cases}\label{Eq:ScheduleRelation_FirstBest}
\end{align}
This shows the temporal-sorting property under the first-best policy.

\section{Closed-form solution of the first-best equilibrium}\label{Sec:App-ClosedFormSolution_FirstBest}
First of all, observing the relation in Eq.~\eqref{Eq:ScheduleRelation_FirstBest}, we derive the following corollary:
\begin{coro}
	In the first-best equilibrium, a flow pattern consisting solely of SAVs cannot occur; some NV traffic must always be present.	
\end{coro}

\noindent This corollary shows that the first-best equilibrium is either (i) a mixed NV-SAV pattern or (ii) an all-NV pattern.
We then derive the closed-form expressions for each case separately.

\subsection{Case (i): mixture of normal and shared autonomous vehicles}
In this case, the flows of NV and SAV commuters alternate at the instants when the schedule cost reaches a prescribed value.
From the complementarity conditions~\eqref{Eq:FirstBest_Optimality1} and ~\eqref{Eq:FirstBest_Optimality2} and the schedule-cost relation~\eqref{Eq:ScheduleRelation_FirstBest}, we define the following critical times:
\begin{align}
	&t^{-}_{n} = \argmin_{t\in\mathbb{R}}\left\{\, t \mid s(t) = c^{FB*} - (t_{f} + F_{n}) \right\},\quad 
	t^{+}_{n} = \argmax_{t\in\mathbb{R}}\left\{\, t \mid s(t) = c^{FB*} - (t_{f} + F_{n}) \right\},\\
	&t^{-}_{a} = \argmin_{t\in\mathbb{R}}\left\{\, t \mid s(t) = c^{FB*} - (t_{f} + F_{n}) - \cfrac{B}{1-\kappa} \right\},\quad 
	t^{+}_{a} = \argmax_{t\in\mathbb{R}}\left\{\, t \mid s(t) = c^{FB*} - (t_{f} + F_{n}) - \cfrac{B}{1-\kappa} \right\},
\end{align}
where $t_{n}^{-}$ and $t_{n}^{+}$ are the earliest and latest destination arrival times of NV commuters; $t_{a}^{-}$ and $t_{a}^{+}$ are those of SAV commuters.

We obtain the following lemma about the flow pattern in the first-best equilibrium:
\begin{lemm}
	The patterns of destination arrival flow of NV and SAV commuters are given as follows:
	\begin{align}
		&n^{*}_{n}(t) = 
		\begin{cases}
			\mu	\quad &t\in(t^{-}_{n}, t_{a}^{-}]\quad \text{or}\quad t\in[t^{+}_{a}, t_{n}^{+})\\
			0	\quad &\text{otherwise}
		\end{cases}\\
		&n^{*}_{a}(t) = 
		\begin{cases}
			\mu/ \kappa	\quad &t\in(t^{-}_{a}, t_{a}^{+})\\
			0	\quad \text{otherwise}
		\end{cases}
	\end{align}
\end{lemm}
\begin{proof}
	First, consider the time periods $t\leq t_{n}^{-}$ or $t\geq t_{n}^{+}$.
	During these periods, we have
	\begin{align}
		s(t) \geq c^{FB*} - (t+F_{n}).
	\end{align}
	Because $B>0$, the left-hand side of the complementartity condition~\eqref{Eq:FirstBest_Optimality2} can be rewitten as follows:
	\begin{align}
		s(t) + \theta t_{f} + m  + w  + \kappa \tau(t) > s(t) + t_{f} + F_{n} + \kappa \tau(t) \geq c^{FB*} + \kappa \tau(t) \geq  c^{FB*}.
	\end{align}
	Therefore, the complementarity condition suggests that $n_{a}(t) = 0$.
	Moreover, the left-hand side of the the complementartity condition~\eqref{Eq:FirstBest_Optimality1} can be rewitten as follows:
	\begin{align}
		s(t) + t_{f} + F_{n} + \tau(t) \geq c^{FB*} + \tau(t) \geq c^{FB*},
	\end{align}
	and equality holds only when $t=t^{-}_{n}$ or $t=t^{+}_{n}$.
	Hence, for $t < t^{-}_{n}$ and $t > t_{n}^{+}$, we obtain $n_{n}(t) = 0$ from the complementarity condition.
	At the isolated points $t = t_{n}^{-}$, $t_{n}^{+}$, we can set $n_{n}(t) = 0$ without affecting the subsequent analysis.
	Therefore, in thse periods, $n_{n}(t) = n_{a}(t) = 0$.

	Next, consider the time periods $t^{-}_{n} < t \leq t^{-}_{a}$ and $t^{+}_{a} \leq t < t^{+}_{n}$.
	During these periods, we have
	\begin{align}
		c^{FB*} - (t+F_{n}) > s(t) \geq c^{FB*} - (t_{f} + F_{n}) - \cfrac{B}{(1-\kappa)}.
	\end{align}
	From Eq.~\eqref{Eq:ScheduleRelation_FirstBest}, we see that $n_{a}(t) = 0$.
	Thus, only NV flow is feasible.

	We prove that $n_{n}(t) = \mu$ by contradiction.
	Suppose that $n_{n}(t) < \mu$.
	Then, $\tau(t) = 0$.
	However, this leads to the following contradiction.
	\begin{align}
		s(t) + t_{f} + F_{n} + \tau(t)  
		= s(t) + t_{f} + F_{n}  < c^{FB*}.
	\end{align}
	Therefore, $n_{n}(t) = \mu$.

	Finally, consider the time period $t_{a}^{-} < t < t_{a}^{+}$.
	During this period, we have
	\begin{align}
		c^{FB*} - (t_{f} + F_{n}) - \cfrac{B}{(1-\kappa)} > s(t) \geq 0,
	\end{align}
	Thus, only SAV flow is feasible.
	
	We prove that $n_{a}(t) = \mu/\kappa$ by contradiction.
	Suppose that $n_{n}(t) < \mu/\kappa$.
	Then, $\tau(t) = 0$.
	However, this leads to the following contradiction.
	\begin{align}
		s(t) + \theta t_{f} + m + w  + \kappa \tau(t)  
		= s(t) + \theta t_{f} + m  + w   < c^{FB*} + B-\cfrac{B}{1-\kappa} < c^{FB*}.
	\end{align}
	Therefore, $n_{a}(t) = \mu/\kappa$.
\end{proof}


Using the destination arrival flow and the earliest/latest destination arrival times of NV and SAV commuters, their total numbers $N_{n}$ and $N_{a}$ are written as follows:
\begin{align}
	&N_{a} = \cfrac{\mu}{\kappa}(t_{a}^{+} - t_{a}^{-}) = \cfrac{\mu}{\kappa}\cfrac{\beta + \gamma}{\beta \gamma}\left[c^{FB*} - (t_{f} + F_{n}) - \cfrac{B}{1-\kappa}\right]\\
	&N_{n} = \mu[t_{n}^{+} - t_{n}^{-} -(t_{a}^{+} - t_{a}^{-}) ]
	=\mu \cfrac{\beta + \gamma}{\beta \gamma}\cfrac{B}{1-\kappa}
\end{align}
Because $N_{n} + N_{a} = N$, we have the following relationship:
\begin{align}
	N = N_{a} + N_{n}
	= \mu \cfrac{\beta + \gamma}{\beta \gamma}\left[\cfrac{c^{FB*} - (t_{f} + F_{n})}{\kappa} -  \cfrac{B}{\kappa} \right]
\end{align}
Solving this equation, we derive the Lagrange multiplier $c^{FB*}$ as follows:
\begin{align}
	c^{FB*} 
	= \cfrac{\kappa AN}{(1-\theta)} + B + t_{f} + F_{n}.
\end{align}
Substituting $c^{FB*}$ into the equation, we can derive the Lagrange multiplier $\tau(t)$.

Finally, substituting $c^{FB*}$ into the equation, we have the earliest/latest destination arrival times, as follows:
\begin{align}
	&t_{n}^{-} = - \cfrac{\kappa AN + (1-\theta)B}{\beta(1-\theta)}
	,\quad 
	t_{n}^{+} = \cfrac{\kappa AN + (1-\theta)B}{\gamma(1-\theta)},\\
	&t_{a}^{-} = - \cfrac{\kappa AN + (1-\theta)B}{\beta(1-\theta)} + \cfrac{B}{\beta (1-\kappa)},\quad 
	t_{a}^{+} = \cfrac{\kappa AN + (1-\theta)B}{\gamma(1-\theta)} - \cfrac{B}{\gamma (1-\kappa)},
\end{align}

\subsection{Case (ii): normal vehicles only}
In this case, the first-best equilibrium coincides with the dynamic system-optimal solution of the classical homogeneous bottleneck model. 
Accordingly, the optimal flows and the associated Lagrange multipliers are given as follows:
\begin{align}
	&n^{*}_{n}(t) = 
	\begin{cases}
		\mu	\quad &t\in (t^{-},t^{+}),\\
		0	\quad &\text{otherwise},
	\end{cases}\\
	&n^{*}_{a}(t) = 0,\quad \forall t\in\mathbb{R},\\
	&c^{FB*} = \cfrac{AN}{1-\theta}+ t_{f} + F_{n},
\end{align}
where $t^{-}$ and $t^{+}$ denote the earliest and latest destination arrival times of commuters, respectively.
They are represented as
\begin{align}
	t^{-} = - \cfrac{\gamma}{\beta + \gamma}\cfrac{N}{\mu},\quad t^{+} = \cfrac{\beta}{\beta + \gamma}\cfrac{N}{\mu}.
\end{align}

We finally derive the condition under which SAVs are used by examining the following relationship:
\begin{align}
c^{FB*} - (t_{f} + F_{n}) - \cfrac{B}{1-\kappa} > 0.\label{Eq:FirstBest_MixturePremise}
\end{align}
As Eq.~\eqref{Eq:ScheduleRelation_FirstBest} indicates, this inequality is a necessary and sufficient condition for the SAV flow to be strictly positive.
Substituting the expression for $c^{FB*}$ into Eq.~\eqref{Eq:FirstBest_MixturePremise}, we derive the condition \eqref{Eq:FirstBest_PopulationThres}.

\section{Pareto improvement property of the first-best policy}\label{Sec:App-ParetoImprovement}
First of all, using the index $\eta$, we can determine which case arises in the first-best equilibrium:
\begin{prop}
If $\eta>1$, the first-best equilibrium necessarily consists of both NV and SAV commuters; that is, case (i) arises.
\end{prop}
\begin{proof}
	Using the index $\eta$, the condition~\eqref{Eq:FirstBest_PopulationThres} can be rewritten as follows:
	\begin{align}
		(AN-B) + (\eta - 1)AN > 0.
	\end{align}
	Here, the first term on the left-hand side of the equation is positive because of the assumption $AN-B>0$.
	In addition, when $\eta>1$, the second term is always positive since $(\eta-1)>0$ and $AN > 0$.
	This means that the left-hand side is always positive and the condition always holds when $\eta\geq 1$.
	Therefore, both NV and SAV are used in the first-best equilibrium.
\end{proof}

We then compare the equilibrium commuting costs under the first-best policy and MC pricing.
Let $c^{FB*}_{1}$ and $c^{FB*}_{2}$ denote the Lagrange multipliers for cases (i) and (ii).
As discussed earlier, each $c^{FB*}$ represents the equilibrium commuting cost under the first-best policy.
By comparing them with the equilibrium commuting cost under the MC-pricing $c^{MC*}$, we have the following relations:
\begin{itemize}
	\item If $\eta \geq 1$:
		\begin{align}
			c^{MC*} - c^{FB*}_{1} = (\eta - 1)B \geq 0.
		\end{align}
		
	\item If $\eta < 1$: 
		\begin{align}
			c^{MC*} - c^{FB*}_{1} < 0,\quad c^{MC*} - c^{FB*}_{2} = -\eta (AN-B) < 0.
		\end{align}
\end{itemize}
This means that, when $\eta \geq 1$, the equilibrium commuting cost under the first-best policy is always equal to or lower than that under MC pricing, whereas for $\eta < 1$, the equilibrium commuting costs under the first-best policy in both cases exceed that under MC pricing.
We thus obtain the proposition about the Pareto improvement property.

\section{Social cost comparison under the three fare-setting scenarios}\label{Sec:App-SC_Comparison}
This section conducts pairwise comparisons of the social costs under the three fare-setting scenarios. 
By summarizing these comparison results, we establish the propositions.

\subsection{MC pricing vs. AC pricing}

\subsubsection{Comparison between $SC^{MC*}$ and $SC_{2}^{AC*}$}
First, we compare $SC^{MC*}$ with $SC_{2}^{AC*}$.
The difference is given by
\begin{align}
	SC^{AC*}_{2} - SC^{MC*}
	=\cfrac{\eta}{2}(AN-B - K)N -F_{a},
\end{align}
where $K = \sqrt{(AN-B)^{2} - 4AF_{a}}$.
We observe that
\begin{align}
	AN-B-K = \cfrac{(AN-B)^{2} - K^{2}}{AN-B+K} = \cfrac{4AF_{a}}{AN-B+K} > 0.
\end{align}
Using this, the difference can be rewritten as
\begin{align}
	SC^{AC*}_{2} - SC^{MC*}
	&=F_{a} \left[ \eta \cdot \cfrac{2AN}{AN-B+K} - 1  \right].
\end{align}
Because we assume $F_{a}>0$, the sign of the cost difference is determined by
\begin{align}
	&SC^{AC*}_{2} - SC^{MC*} = 
	\begin{cases}
		>0\quad &\text{if}\quad G(N) > 0,\\
		0\quad &\text{if}\quad G(N) = 0,\\
		<0\quad &\text{if}\quad G(N) < 0,
	\end{cases}\\
	&\text{where}\quad G(N) := \eta \cfrac{2AN}{AN-B+K} - 1.
\end{align}
Hence, determining whether $SC^{AC*}_{2} - SC^{MC*}$ is positive, zero, or negative is equivalent to examining the sign of $G(N)$.

Because $N\in (N_{\mathrm{min}}, \infty)$, the function $G(N)$ attains the following boundary values:
\begin{align}
	&G(N_{\mathrm{min}}) = \eta\cdot \left(2 + \cfrac{B}{\sqrt{AF_{a}}}\right) - 1,\\
	&G(N)\xrightarrow{N\to\infty}\eta  - 1.
\end{align}
Moreover, differentiating $G$ with respect to $N$ yields
\begin{align}
	\cfrac{\partial G}{\partial N} 
	= \eta \cfrac{2A(K-AN)}{K(AN-B+K)} < 0,
\end{align}
so $G$ is strictly decreasing in $N$.
Using these results, the sign of $G(N)$ is classified according to the value of $\eta$, as follows:
\begin{itemize}
	\item Case 1: $\eta \geq 1$.
		Because $G(\infty) \geq 0$ and $G$ is strictly decreasing in $N$, we have
		\begin{align}
			G(N) > 0,\quad \text{for every $N\in (N_{\mathrm{\min}}, \infty)$.}
		\end{align}
		
	\item Case 2: $1 > \eta \geq 1/2$.
		Here, $G(\infty) < 0$, while
		\begin{align}
			G(N_{\mathrm{\min}}) = \eta (2 + \cfrac{B}{\sqrt{A F_{a}}}) - 1 > 0.
		\end{align}
		Consequently, there exists a unique critical population size $N = N_{c}^{MC=AC}$ sataisfying $G(N_{c}^{MC=AC}) = 0$, and 
		\begin{align}
			G(N) = 
			\begin{cases}
				>0\quad &\text{if}\quad N_{\mathrm{min}} < N < N_{c}^{MC=AC},\\
				0\quad &\text{if}\quad N = N_{c}^{MC=AC},\\
				<0\quad &\text{if}\quad N_{c}^{MC=AC} < N.
			\end{cases}\label{Eq:App-MCvsAC_G_mediumEta}
		\end{align}

	\item Case 3: $1/2 > \eta$.
		Again, $G(\infty) < 0$. 
		Whether $G(N_{\mathrm{min}})$ is positive or negative depends on $F_{a}$.
		Setting $G(N_{\mathrm{min}}) = 0$ and solving for $F_{a}$ gives the critical fixed cost
		\begin{align}
			F_{a,c}^{MC-AC} = \cfrac{\eta^{2}B^{2}}{(1-2\eta)^{2}A}\quad (>0).
		\end{align}
		\begin{itemize}
			\item If $F_{a} \geq  F_{a,c}^{MC-AC}$, $G(N) \leq 0$ for all $N$, with equality only when $F_{a} =  F_{a,c}^{MC-AC}$ and $N = N_{\mathrm{min}}$.
			\item If $F_{a} < F_{a,c}^{MC-AC}$, we have $G(N_{\mathrm{min}}) > 0$.
				By the monotonicity, a unique critical value $N = N_{c}^{MC=AC}$ exists, and the sign pattern of $G$ is identical to Eq.~\eqref{Eq:App-MCvsAC_G_mediumEta}
		\end{itemize}
\end{itemize}

To determine the critical population size $N_{c}^{MC=AC}$, begin by rewritten the condition $G(N) = 0$:
\begin{align}
	(2\eta - 1)AN + B = \sqrt{(AN-B)^{2}-4AF_{a}},
\end{align}
where the solution must satisfy $(2\eta - 1)AN + B \geq 0$.
Squaring both sides yields the quadratic equation
\begin{align}
	\eta(\eta-1)A N^{2} + \eta B N + F_{a} = 0.
\end{align}
Solving this quadratic gives the candidate critical values
\begin{align}
	N_{c} = \cfrac{\eta B \pm \sqrt{\eta^{2}B^{2} + 4\eta(1-\eta)AF_{a}}}{2\eta(1-\eta)A}.
\end{align}
For the relevant range $0 < \eta < 1$, the expression under the radical is positive, so both roots are real. 
Because 
\begin{align}
	\sqrt{\eta^{2}B^{2} + 4\eta(1-\eta)AF_{a}} > \eta B,
\end{align}
only the plus root is positive.
Hence, the unique admissible critical population size is 
\begin{align}
	N_{c}^{MC=AC} =  \cfrac{\eta B + \sqrt{\eta^{2}B^{2} + 4\eta(1-\eta)AF_{a}}}{2\eta(1-\eta)A}.
\end{align}
Note that this critical population size satisfies the relation $(2\eta - 1)AN_{c} + B\geq 0$.

\subsubsection{Comparison between $SC^{MC*}$ and $SC_{0}^{AC*}$}
We compare $SC^{MC*}$ with $SC_{0}^{AC*}$.
The difference is given as follows:
\begin{align}
	SC^{AC*}_{0} - SC^{MC*} 
	&= \eta(AN-B)N > 0.
\end{align}
Hence, $SC^{MC*} < SC_{0}^{AC*}$.

\subsection{MC pricing vs. unregulated monopoly pricing}
We compare $SC^{MC*}$ with $SC^{m*}$.
The difference is given by
\begin{align}
	SC^{m*} - SC^{MC*}  
	= \cfrac{AN-B}{2}\left( \cfrac{(2\eta -1)AN + B}{2A}  \right)
\end{align}
Because we assume $AN-B > 0$, the sign of the cost difference is determined by
\begin{align}
	&SC^{m*} - SC^{MC*}  = 
	\begin{cases}
		>0\quad &\text{if}\quad G(N) > 0,\\
		0\quad &\text{if}\quad G(N) = 0,\\
		<0\quad &\text{if}\quad G(N) < 0,
	\end{cases}\\
	&\text{where}\quad G(N) := (2\eta -1)AN + B.
\end{align}

The function $G(N)$ attains the following boundary value:
\begin{align}
	G(N_{\mathrm{min}}) 
	= (2\eta -1)\sqrt{4AF_{a}} + 2\eta B
\end{align}
Moreover, differentiating $G$ with respect to $N$ yields
\begin{align}
	\cfrac{\partial G}{\partial N} 
	= (2\eta -1)A,
\end{align}
This means that the sensitivity of $G$ depends on the value of $\eta$.

Using these results, the sign of $G(N)$ is classified, as follows:
\begin{itemize}
	\item Case 1: $\eta > 1/2$.
		Obviously, $G(N_{\mathrm{min}}) > 0$ and $G$ is strictly increasing.
		We thus have
		\begin{align}
			G(N) > 0,\quad \text{for every $N\in (N_{\mathrm{\min}}, \infty)$.}
		\end{align}
		
	\item Case 2: $\eta = 1/2$.
		Here, $G(N_{\mathrm{min}}) = B > 0$, and the sensitivity of $G$ becomes zero.
		We thus have
		\begin{align}
			G(N) > 0,\quad \text{for every $N\in (N_{\mathrm{\min}}, \infty)$.}
		\end{align}

	\item Case 3: $1/2 > \eta$.
		Because $G$ is strictly decreasing, the following statements hold: if $G(N_{\mathrm{min}}) < 0$, then $G(N) < 0$ for every $N\in[N_{\mathrm{min}}, \infty)$;
		otherwise, there exists a unique critical value $N = N_{c}^{MC = m}$ such that $G(N_{c}^{MC=m})$.

		Let us derive the critical fixed cost $F_{a} = F_{a,c}^{MC-m}$ satisfying $G(N_{\mathrm{min}})=0$.
		Setting $G(N_{\mathrm{min}})=0$, we have
		\begin{align}
			 2\eta B = (1 - 2\eta)\sqrt{4AF_{a}}.
		\end{align}
		Since $\eta < 1/2$, the right-hand side is positive;
		squaring both sides and solving for $F_{a}$ yields
		\begin{align}
			F_{a,c}^{MC-m} = \cfrac{\eta^{2}B^{2}}{(1-2\eta)^{2}A}.
		\end{align}
		Note that $F_{a,c}^{MC-m} = F_{a,c}^{MC-AC}$.

		\begin{itemize}
			\item If $F_{a} \geq F_{a,c}^{MC-m}$, then $G(N)\leq 0$, with equality only when $F_{a} = F_{a,c}^{MC-m}$ and $N = N_{\mathrm{min}}$.
			\item If $F_{a} <  F_{a,c}^{MC-m}$, then $G(N_{\mathrm{min}}) > 0$.
				By the monotonicity, there is exactly one critical population size $N = N_{c}^{MC = m}$.
				Accordingly,
				\begin{align}
					G(N) = 
					\begin{cases}
						>0\quad &\text{if}\quad N_{\mathrm{min}}\leq N < N_{c}^{MC=m},\\
						0\quad &\text{if}\quad N = N_{c}^{MC=m},\\
						<0\quad &\text{if}\quad N_{c}^{MC=m} < N.
					\end{cases}
				\end{align}
				Solving the equation $G(N) = 0$ gives the following critical value:
				\begin{align}
					N_{c}^{MC=m} = \cfrac{B}{(1-2\eta)A} \quad (> 0).
				\end{align}
		\end{itemize}
		
\end{itemize}

\subsection{AC pricing vs. unregulated monopoly pricing}

\subsubsection{Comparison between $SC^{m*}$ and $SC_{2}^{AC*}$}
First, we compare $SC^{m*}$ with $SC_{2}^{AC*}$.
The difference is given by
\begin{align}
	SC^{m*} - SC^{AC*}_{2}
	& = \cfrac{K(2\eta AN - K)}{4A}.
\end{align}
Since we assume that $N > N_{\mathrm{\min}}$, $K > 0$.
Hence, 
\begin{align}
	&SS^{AC*}_{2} - SS^{m*}  = 
	\begin{cases}
		>0\quad &\text{if}\quad G(N) > 0,\\
		0\quad &\text{if}\quad G(N) = 0,\\
		<0\quad &\text{if}\quad G(N) < 0,
	\end{cases}\\
	&\text{where}\quad G(N) := 2\eta AN - K.
\end{align}

Let us examine the properties of $G(N)$.
We first obtain 
\begin{align}
	G(N_{\mathrm{min}}) 
	= 2\eta \left( B + \sqrt{4AF_{a}} \right)>0.
\end{align}
Differentiating $G$ with respect to $N$ gives
\begin{align}
	\cfrac{\partial G}{\partial N} 
	= A\left( 2\eta -  \cfrac{AN-B}{K} \right),
\end{align}
so the sensitivity of $G$ depands on the value of $\eta$.

Building on the above results, we can classify the sign of $G(N)$ according to the value of $\eta$.
\begin{itemize}
	\item Case 1: $\eta \geq 1/2$.
		Using the fact that $AN-B > K$, we obtain
		\begin{align}
			G(N) \geq AN - K = AN - B - K + B > B > 0.
		\end{align}
		Thus, $G(N)>0$ for every admissible $N$.
		
	\item Case 2: $1/2 > \eta$.
		In this case,
		\begin{align}
			\cfrac{\partial G}{\partial N}
			= A\left( 2\eta -  \cfrac{AN-B}{K} \right)
			< A\left( 1 - \cfrac{AN-B}{K} \right) < 0.
		\end{align}
		In addition,
		\begin{align}
			G(N) = N\left( 2\eta A - \sqrt{\left(A - \cfrac{B}{N}  \right)^{2} - \cfrac{4AF_{a}}{N^{2}}}\right)
			\xrightarrow{N\to\infty} \infty \cdot (2\eta - 1) A,
		\end{align}
		Because $2\eta - 1 < 0$, it follows that $G(\infty)\rightarrow -\infty$.
		
		Hence, $G(N)$ decreases monotonically from the positive value $G(N\mathrm{\min}) > 0$ to negative values, so there exists a unique critical population size $N_{c}^{AC=m}$ at which $G(N_{c}^{AC=m}) = 0$.
		Consequently,
		\begin{align}
			G(N) = 
			\begin{cases}
				>0\quad &\text{if}\quad N_{\mathrm{min}}\leq N < N_{c}^{AC=m},\\
				0\quad &\text{if}\quad N = N_{c}^{AC=m},\\
				<0\quad &\text{if}\quad N_{c}^{AC=m} < N.
			\end{cases}
		\end{align}

		Let us derive the critical value $G(N_{c}^{AC=m}) = 0$.
		Setting $G(N) = 0$ and rewritting the condition yields
		\begin{align}
			&2\eta AN = K\\
			&\Rightarrow \quad (1 - 4\eta ^{2})A^{2}N^{2} - 2ABN  + B^{2} - 4AF_{a} = 0\\
			&\Rightarrow \quad N = \cfrac{B \pm \sqrt{4\eta^{2}B^{2} +4(1-4\eta^{2})AF_{a})}}{(1-4\eta^{2})A}.
		\end{align}
		Because the critical point must be unique, the smaller root cannot serve: if it did, the larger root would also satisfy $G(N) = 0$, contradicting uniqueness.
		Therefore, the critical value is 
		\begin{align}
			N_{c}^{AC=m} = \cfrac{B + \sqrt{4\eta^{2}B^{2} +4(1-4\eta^{2})AF_{a}}}{(1-4\eta^{2})A},
		\end{align}
\end{itemize}

\subsubsection{Comparison between $SC^{m*}$ and $SC_{0}^{AC*}$}
We compare $SC^{AC*}_{0}$ with $SC^{m*}$.
The difference is given as follows:
\begin{align}
	SC^{m*} - SC^{AC*}_{0} 
	& =  - \cfrac{(AN - B)^{2}}{4A} 
	- \cfrac{\eta}{2}(AN-B)N < 0
\end{align}
Hence, $SC^{m*} < SC_{0}^{AC*}$.

\bibliographystyle{elsarticle-harv} 
\bibliography{BNAnalysis.bib}

@article{ying2005sensitivity,
  title = {Sensitivity Analysis of Stochastic User Equilibrium Flows in a Bi-Modal Network with Application to Optimal Pricing},
  author = {Ying, Jiang Qian and Yang, Hai},
  year = {2005},
  month = nov,
  journal = {Transportation Research Part B: Methodological},
  volume = {39},
  number = {9},
  pages = {769--795},
  issn = {0191-2615},
  doi = {10.1016/j.trb.2003.09.003},
  urldate = {2025-08-06}
}

@ARTICLE{Ma2017-vg,
  title    = "The morning commute problem with ridesharing and dynamic parking
              charges",
  author   = "Ma, Rui and Zhang, H M",
  journal  = "Transportation Research Part B: Methodological",
  volume   =  106,
  pages    = "345--374",
  month    =  dec,
  year     =  2017
}

@ARTICLE{Fosgerau2018-ju,
  title     = "Vickrey meets Alonso: Commute scheduling and congestion in a
               monocentric city",
  author    = "Fosgerau, Mogens and Kim, Jinwon and Ranjan, Abhishek",
  journal   = "Journal of Urban Economics",
  publisher = "Elsevier BV",
  volume    =  105,
  pages     = "40--53",
  month     =  may,
  year      =  2018,
  language  = "en"
}

@ARTICLE{Arnott2010-ei,
  title    = "The stability of downtown parking and traffic congestion",
  author   = "Arnott, Richard and Inci, Eren",
  journal  = "Journal of urban economics",
  volume   =  68,
  number   =  3,
  pages    = "260--276",
  month    =  nov,
  year     =  2010
}

@ARTICLE{Liu2018-pa,
  title    = "An equilibrium analysis of commuter parking in the era of
              autonomous vehicles",
  author   = "Liu, Wei",
  journal  = "Transportation Research Part C: Emerging Technologies",
  volume   =  92,
  pages    = "191--207",
  month    =  jul,
  year     =  2018
}

@article{wu2023managing,
  title = {Managing a Bi-Modal Bottleneck System with Manned and Autonomous Vehicles: Incorporating the Effects of in-Vehicle Activity Utilities},
  shorttitle = {Managing a Bi-Modal Bottleneck System with Manned and Autonomous Vehicles},
  author = {Wu, Suping and Li, Zhi-Chun},
  year = {2023},
  month = jul,
  journal = {Transportation Research Part C: Emerging Technologies},
  volume = {152},
  pages = {104179},
  issn = {0968-090X},
  doi = {10.1016/j.trc.2023.104179},
  urldate = {2025-08-06}
}

@article{li2022can,
  title = {Can Day-to-Day Dynamic Model Be Solved Analytically? {New} Insights on Portraying Equilibrium and Accommodating Autonomous Vehicles},
  shorttitle = {Can Day-to-Day Dynamic Model Be Solved Analytically?},
  author = {Li, Pengbo and Tian, Lijun and Xiao, Feng and Zhu, Hongwei},
  year = {2022},
  month = dec,
  journal = {Transportation Research Part B: Methodological},
  volume = {166},
  pages = {374--395},
  issn = {0191-2615},
  doi = {10.1016/j.trb.2022.11.003},
  urldate = {2025-08-06}
}

@article{lamotte2017use,
  title = {On the Use of Reservation-Based Autonomous Vehicles for Demand Management},
  author = {Lamotte, Rapha{\"e}l and {de Palma}, Andr{\'e} and Geroliminis, Nikolas},
  year = {2017},
  month = may,
  journal = {Transportation Research Part B: Methodological},
  volume = {99},
  pages = {205--227},
  issn = {0191-2615},
  doi = {10.1016/j.trb.2017.01.003},
  urldate = {2024-11-25}
}

@article{pudane2020departure,
  title = {Departure Time Choice and Bottleneck Congestion with Automated Vehicles: Role of On-Board Activities},
  shorttitle = {Departure Time Choice and Bottleneck Congestion with Automated Vehicles},
  author = {Pud{\=a}ne, Baiba},
  year = {2020},
  month = oct,
  journal = {European Journal of Transport and Infrastructure Research},
  volume = {20},
  number = {4},
  pages = {306--334},
  issn = {1567-7141},
  doi = {10.18757/ejtir.2020.20.4.4801},
  urldate = {2025-08-06},
  copyright = {Copyright (c) 2020 Baiba Pudane},
  langid = {english}
}

@article{fielbaum2024are,
  title = {Are Shared Automated Vehicles Good for Public- or Private-Transport-Oriented Cities (or Neither)?},
  author = {Fielbaum, Andr{\'e}s and Pud{\=a}ne, Baiba},
  year = {2024},
  month = nov,
  journal = {Transportation Research Part D: Transport and Environment},
  volume = {136},
  pages = {104373},
  issn = {1361-9209},
  doi = {10.1016/j.trd.2024.104373},
  urldate = {2025-01-03}
}

@article{fielbaum2023economies,
  title = {Economies and Diseconomies of Scale in On-Demand Ridepooling Systems},
  author = {Fielbaum, Andr{\'e}s and Tirachini, Alejandro and {Alonso-Mora}, Javier},
  year = {2023},
  month = jun,
  journal = {Economics of Transportation},
  volume = {34},
  pages = {100313},
  issn = {2212-0122},
  doi = {10.1016/j.ecotra.2023.100313},
  urldate = {2024-12-28}
}

@article{li2016dynamics,
  title = {Dynamics of Modal Choice of Heterogeneous Travelers with Responsive Transit Services},
  author = {Li, Xinwei and Yang, Hai},
  year = {2016},
  month = jul,
  journal = {Transportation Research Part C: Emerging Technologies},
  volume = {68},
  pages = {333--349},
  issn = {0968-090X},
  doi = {10.1016/j.trc.2016.04.014},
  urldate = {2025-08-05}
}

@article{li2018traffic,
  title = {Traffic Dynamics in a Bi-Modal Transportation Network with Information Provision and Adaptive Transit Services},
  author = {Li, Xinwei and Liu, Wei and Yang, Hai},
  year = {2018},
  month = jun,
  journal = {Transportation Research Part C: Emerging Technologies},
  volume = {91},
  pages = {77--98},
  issn = {0968-090X},
  doi = {10.1016/j.trc.2018.03.026},
  urldate = {2025-08-05}
}

@article{cantarella2015daytoday,
  title = {Day-to-Day Dynamics \& Equilibrium Stability in A Two-Mode Transport System with Responsive Bus Operator Strategies},
  author = {Cantarella, Giulio E. and Velon{\`a}, Pietro and Watling, David P.},
  year = {2015},
  month = sep,
  journal = {Networks and Spatial Economics},
  volume = {15},
  number = {3},
  pages = {485--506},
  issn = {1572-9427},
  doi = {10.1007/s11067-013-9188-4},
  urldate = {2025-08-05},
  langid = {english}
}

@article{iryo2019properties,
  title = {Properties of Equilibria in Transport Problems with Complex Interactions between Users},
  author = {Iryo, Takamasa and Watling, David},
  year = {2019},
  month = aug,
  journal = {Transportation Research Part B: Methodological},
  volume = {126},
  pages = {87--114},
  issn = {0191-2615},
  doi = {10.1016/j.trb.2019.05.006},
  urldate = {2025-08-05}
}

@article{wang2019optimal,
  title = {Optimal Parking Supply in Bi-Modal Transportation Network Considering Transit Scale Economies},
  author = {Wang, Jing and Zhang, Xiaoning and Wang, Hua and Zhang, Michael},
  year = {2019},
  month = oct,
  journal = {Transportation Research Part E: Logistics and Transportation Review},
  volume = {130},
  pages = {207--229},
  issn = {1366-5545},
  doi = {10.1016/j.tre.2019.09.003},
  urldate = {2025-08-05}
}

@incollection{bell2012road,
  title = {Road Use Charging and Inter-Modal User Equilibrium: The Downs-Thompson Paradox Revisited},
  shorttitle = {Road Use Charging and Inter-Modal User Equilibrium},
  booktitle = {Energy, Transport, \& the Environment: Addressing the Sustainable Mobility Paradigm},
  author = {Bell, Michael G. H. and Wichiensin, Muanmas},
  editor = {Inderwildi, Oliver and King, Sir David},
  year = {2012},
  pages = {373--383},
  publisher = {Springer},
  address = {London},
  urldate = {2025-08-05},
  isbn = {978-1-4471-2717-8},
  langid = {english}
}

@article{basso2012integrating,
  title = {Integrating Congestion Pricing, Transit Subsidies and Mode Choice},
  author = {Basso, Leonardo J. and {Jara-D{\'i}az}, Sergio R.},
  year = {2012},
  month = jul,
  journal = {Transportation Research Part A: Policy and Practice},
  series = {``Transportation Economics''},
  volume = {46},
  number = {6},
  pages = {890--900},
  issn = {0965-8564},
  doi = {10.1016/j.tra.2012.02.013},
  urldate = {2025-08-05}
}

@article{arnott2000twomode,
  title = {The Two-Mode Problem: Second-Best Pricing and Capacity},
  shorttitle = {The Two-Mode Problem},
  author = {Arnott, Richard and Yan, An},
  year = {2000},
  journal = {Review of Urban \& Regional Development Studies},
  volume = {12},
  number = {3},
  pages = {170--199},
  issn = {1467-940X},
  doi = {10.1111/j.1467-940X.2000.00077.x},
  urldate = {2025-08-05},
  langid = {english}
}

@article{zhang2014downs,
  title = {The Downs--Thomson Paradox with Responsive Transit Service},
  author = {Zhang, Fangni and Yang, Hai and Liu, Wei},
  year = {2014},
  month = dec,
  journal = {Transportation Research Part A: Policy and Practice},
  volume = {70},
  pages = {244--263},
  issn = {0965-8564},
  doi = {10.1016/j.tra.2014.10.022},
  urldate = {2023-10-17}
}

@article{Danielis2002,
  title = {Bottleneck Road Congestion Pricing with a Competing Railroad Service},
  author = {Danielis, Romeo and Marcucci, Edoardo},
  year = {2002},
  journal = {Transportation Research Part E: Logistics and Transportation Review},
  volume = {38},
  number = {5},
  pages = {379--388},
  issn = {13665545},
  doi = {10.1016/S1366-5545(01)00021-7},
  isbn = {1366-5545}
}

@article{dantsuji2024hypercongestiona,
  title = {Hypercongestion, Autonomous Vehicles, and Urban Spatial Structure},
  author = {Dantsuji, Takao and Takayama, Yuki},
  year = {2024},
  month = aug,
  journal = {Transportation Science},
  volume = {58},
  number = {6},
  pages = {1352--1370},
  publisher = {INFORMS},
  issn = {0041-1655},
  doi = {10.1287/trsc.2024.0519},
  urldate = {2024-08-15},
  copyright = {All rights reserved}
}

@article{tian2019morning,
  title = {The Morning Commute Problem with Endogenous Shared Autonomous Vehicle Penetration and Parking Space Constraint},
  author = {Tian, Li-Jun and Sheu, Jiuh-Biing and Huang, Hai-Jun},
  year = {2019},
  month = may,
  journal = {Transportation Research Part B: Methodological},
  volume = {123},
  pages = {258--278},
  issn = {0191-2615},
  doi = {10.1016/j.trb.2019.04.001},
  urldate = {2025-01-02}
}

@article{yu2022will,
  title = {Will All Autonomous Cars Cooperate? {{Brands}}' Strategic Interactions under Dynamic Congestion},
  shorttitle = {Will All Autonomous Cars Cooperate?},
  author = {Yu, Xiaojuan and {van den Berg}, Vincent A. C. and Verhoef, Erik T. and Li, Zhi-Chun},
  year = {2022},
  month = oct,
  journal = {Transportation Research Part E: Logistics and Transportation Review},
  volume = {166},
  pages = {102825},
  issn = {1366-5545},
  doi = {10.1016/j.tre.2022.102825},
  urldate = {2025-01-03}
}

@article{downs1962law,
  title = {The Law of Peak-Hour Expressway Congestion},
  author = {Downs, Anthony},
  year = {1962},
  journal = {Traffic Quarterly},
  volume = {16},
  number = {3},
  pages = {393--409}
}

@book{thomson1977great,
  title = {Great {{Cities}} and {{Their Traffic}}},
  author = {Thomson},
  year = {1977},
  publisher = {Victor Gollancz},
  address = {London}
}

@book{varian2024intermediate,
  title = {Intermediate {{Microeconomics}}: {{A Modern Approach}}},
  author = {Varian, Hal R and Melitz, Marc J},
  year = {2024},
  publisher = {WW Norton \& Company}
}

@article{horcher2021review,
  title = {A Review of Public Transport Economics},
  author = {H{\"o}rcher, Daniel and Tirachini, Alejandro},
  year = {2021},
  month = mar,
  journal = {Economics of Transportation},
  volume = {25},
  pages = {100196},
  issn = {2212-0122},
  doi = {10.1016/j.ecotra.2021.100196},
  urldate = {2023-10-17}
}

@book{small2024economics,
  title = {The {{Economics}} of {{Urban Transportation}}},
  author = {Small, Kenneth A and Verhoef, Erik T and Lindsey, Robin},
  year = {2024},
  edition = {Third},
  publisher = {Routledge},
  urldate = {2024-11-06},
  langid = {english}
}

@article{narayanan2020shareda,
  title = {Shared Autonomous Vehicle Services: {{A}} Comprehensive Review},
  shorttitle = {Shared Autonomous Vehicle Services},
  author = {Narayanan, Santhanakrishnan and Chaniotakis, Emmanouil and Antoniou, Constantinos},
  year = {2020},
  month = feb,
  journal = {Transportation Research Part C: Emerging Technologies},
  volume = {111},
  pages = {255--293},
  issn = {0968-090X},
  doi = {10.1016/j.trc.2019.12.008},
  urldate = {2024-11-22}
}

@INCOLLECTION{Brown1951-my,
  title     = "Solutions of games by differential equations",
  author    = "Brown, G W and Neumann, J von",
  editor    = "Kuhn, Harold William and Tucker, Albert William",
  booktitle = "Contributions to the Theory of Games (AM-24), Volume I",
  publisher = "Princeton University Press",
  address   = "Princeton",
  pages     = "73--80",
  month     =  dec,
  year      =  1951
}

@ARTICLE{Gilboa1991-xj,
  title     = "Social stability and equilibrium",
  author    = "Gilboa, Itzhak and Matsui, Akihiko",
  journal   = "Econometrica",
  publisher = "JSTOR",
  volume    =  59,
  number    =  3,
  pages     =  859,
  month     =  may,
  year      =  1991
}

@ARTICLE{Pandey2024-xc,
  title     = "Congestive mode-switching and economies of scale on a bus route",
  author    = "Pandey, Ayush and Lehe, Lewis J",
  journal   = "Transportation Research Part B: Methodological",
  publisher = "Elsevier BV",
  volume    =  183,
  number    =  102930,
  pages     =  102930,
  month     =  may,
  year      =  2024,
  language  = "en"
}

@ARTICLE{Arnott1988-fy,
  title   = "Schedule delay and departure time decisions with heterogeneous
             commuters",
  author  = "Arnott, R and Palma, A and Lindsey, R",
  journal = "Transportation Research Record",
  year    =  1988
}

@BOOK{Mohring1962-wz,
  title     = "Highway Benefits: An Analytical Framework",
  author    = "Mohring, Herbert and Harwitz, Mitchell",
  publisher = "Northwestern University Press",
  year      =  1962
}

@ARTICLE{Tian2019-ug,
  title     = "The morning commute problem with endogenous shared autonomous
               vehicle penetration and parking space constraint",
  author    = "Tian, Li-Jun and Sheu, Jiuh-Biing and Huang, Hai-Jun",
  journal   = "Transportation Research Part B: Methodological",
  publisher = "Elsevier BV",
  volume    =  123,
  pages     = "258--278",
  month     =  may,
  year      =  2019,
  language  = "en"
}

@ARTICLE{Wu2014-xg,
  title     = "Equilibrium and modal split in a competitive highway/transit
               system under different road-use pricing strategies",
  author    = "Wu, Wen-Xiang and Huang, Hai-Jun",
  journal   = "Journal of Transport Economics and Policy",
  publisher = "University of Bath",
  volume    =  48,
  number    =  1,
  pages     = "153--169",
  year      =  2014
}

@ARTICLE{Tabuchi1993-xi,
  title    = "Bottleneck congestion and modal split",
  author   = "Tabuchi, Takatoshi",
  journal  = "Journal of Urban Economics",
  volume   =  34,
  number   =  3,
  pages    = "414--431",
  month    =  nov,
  year     =  1993
}

@ARTICLE{Sakai2024-zq,
  title     = "A paradox of telecommuting and staggered work hours in the
               bottleneck model",
  author    = "Sakai, Takara and Akamatsu, Takashi and Satsukawa, Koki",
  journal   = "Transportation Science",
  volume = {58},
  number = {6},
  pages = {1335--1351},
  publisher = "INFORMS",
  month     =  jul,
  year      =  2024
}

@ARTICLE{Smith1984-ed,
  title     = "The stability of a dynamic model of traffic assignment—An
               application of a method of Lyapunov",
  author    = "Smith, Michael J",
  journal   = "Transportation Science",
  publisher = "INFORMS",
  volume    =  18,
  number    =  3,
  pages     = "245--252",
  year      =  1984
}

@ARTICLE{Arnott1990-oa,
  title    = "Economics of a bottleneck",
  author   = "Arnott, Richard and de Palma, André and Lindsey, Robin",
  journal  = "Journal of Urban Economics",
  volume   =  27,
  number   =  1,
  pages    = "111--130",
  month    =  jan,
  year     =  1990
}

@BOOK{Beckmann1956-vr,
  title     = "Studies in the Economics of Transportation",
  author    = "Beckmann, Martin and McGuire, C B and Winsten, Christopher B",
  publisher = "Yale University Press",
  address   = "New Haven",
  month     =  mar,
  year      =  1956
}

@BOOK{Sandholm2010-ht,
  title     = "Population Games and Evolutionary Dynamics",
  author    = "Sandholm, William H",
  publisher = "The MIT Press",
  year      =  2010
}

@ARTICLE{Arnott1993-jd,
  title     = "A structural model of peak-period congestion: A traffic
               bottleneck with elastic demand",
  author    = "Arnott, Richard and de Palma, André and Lindsey, Robin",
  journal   = "The American Economic Review",
  publisher = "American Economic Association",
  volume    =  83,
  number    =  1,
  pages     = "161--179",
  year      =  1993
}

@ARTICLE{Yu2022-dv,
  title    = "Autonomous cars and activity-based bottleneck model: How do
              in-vehicle activities determine aggregate travel patterns?",
  author   = "Yu, Xiaojuan and van den Berg, Vincent A C and Verhoef, Erik T",
  journal  = "Transportation Research Part C: Emerging Technologies",
  volume   =  139,
  pages    =  103641,
  month    =  jun,
  year     =  2022
}

@ARTICLE{Vickrey1969-rg,
  title     = "Congestion theory and transport investment",
  author    = "Vickrey, William S",
  journal   = "The American Economic Review",
  publisher = "American Economic Association",
  volume    =  59,
  number    =  2,
  pages     = "251--260",
  year      =  1969
}

@ARTICLE{Van_den_Berg2016-mv,
  title    = "Autonomous cars and dynamic bottleneck congestion: The effects on
              capacity, value of time and preference heterogeneity",
  author   = "van den Berg, Vincent A C and Verhoef, Erik T",
  journal  = "Transportation Research Part B: Methodological",
  volume   =  94,
  pages    = "43--60",
  month    =  dec,
  year     =  2016
}

@ARTICLE{Osawa2018-ip,
  title    = "First-best dynamic assignment of commuters with endogenous
              heterogeneities in a corridor network",
  author   = "Osawa, Minoru and Fu, Haoran and Akamatsu, Takashi",
  journal  = "Transportation Research Part B: Methodological",
  volume   =  117,
  pages    = "811--831",
  month    =  nov,
  year     =  2018
}

@ARTICLE{Takayama2017-ge,
  title    = "Bottleneck congestion and residential location of heterogeneous
              commuters",
  author   = "Takayama, Yuki and Kuwahara, Masao",
  journal  = "Journal of Urban Economics",
  volume   =  100,
  pages    = "65--79",
  month    =  jul,
  year     =  2017
}

@ARTICLE{Hendrickson1981-bq,
  title     = "Schedule delay and departure time decisions in a deterministic
               model",
  author    = "Hendrickson, Chris and Kocur, George",
  journal   = "Transportation Science",
  publisher = "INFORMS",
  volume    =  15,
  number    =  1,
  pages     = "62--77",
  year      =  1981
}

\end{document}